\newcommand{\References}{references}
\documentclass[a4paper]{article}
\usepackage[utf8]{inputenc}
\usepackage[T1]{fontenc}
\usepackage{amsmath}
\usepackage{amssymb}

\usepackage{pdfpages} 
\usepackage[colorlinks=true]{hyperref}
\usepackage{xcolor}
\usepackage{pst-all}
\usepackage{pstricks-add}
\usepackage{tabto}
\usepackage{leftidx}
\usepackage{xstring}
\usepackage{todonotes}
\usepackage{cite}
\usepackage[shortlabels]{enumitem}
\usepackage{dsfont}
\usepackage{marvosym}
\usepackage{stmaryrd}
\usepackage{multicol}
\usepackage{graphicx}
\usepackage{epsfig}
\usepackage{geometry}
\usepackage{tabularx}
\usepackage{mathtools}
\usepackage{mathrsfs}
\usepackage{float}
\usepackage{subcaption}

\newif\ifCustomTheorems
\usepackage{amsthm}

\makeatletter
\newcommand\RedeclareMathOperator{%
  \@ifstar{\def\rmo@s{m}\rmo@redeclare}{\def\rmo@s{o}\rmo@redeclare}%
}
\newcommand\rmo@redeclare[2]{%
  \begingroup \escapechar\m@ne\xdef\@gtempa{{\string#1}}\endgroup
  \expandafter\@ifundefined\@gtempa
     {\@latex@error{\noexpand#1undefined}\@ehc}%
     \relax
  \expandafter\rmo@declmathop\rmo@s{#1}{#2}}
\newcommand\rmo@declmathop[3]{%
  \DeclareRobustCommand{#2}{\qopname\newmcodes@#1{#3}}%
}
\@onlypreamble\RedeclareMathOperator
\makeatother

\newcommand{\N}{\mathds{N}}
\newcommand{\R}{\mathds{R}}
\newcommand{\Rp}{\R_{\geq0}}
\newcommand{\Rpp}{\R_{>0}}

\newcommand{\Impl}{\Longrightarrow}
\newcommand{\sub}{\subseteq}

\newcommand{\fa}{\ \forall \, }
\newcommand{\ex}{\ \exists \, }

\newcommand{\To}{\longrightarrow}

\newcommand{\rbl}{\left (}
\newcommand{\rbr}{\right )}
\newcommand{\sbl}{\left [}
\newcommand{\sbr}{\right ]}

\newcommand{\bl}{\left |}
\newcommand{\br}{\right |}
\newcommand{\nl}{\left\|}
\newcommand{\nr}{\right\|}
\newcommand{\cbl}{\left\lbrace }
\newcommand{\cbr}{\right\rbrace }
\newcommand{\Abs}[1]{\bl #1 \br}
\newcommand{\Norm}[2][ ]{\nl #2 \nr_{#1}}
\newcommand{\SNorm}[1]{\Norm[\infty]{#1}}
\newcommand{\LNorm}[2][2]{\Norm[L^{#1}]{#2}}

\newcommand{\setdef}[2]{\cbl\ #1\ \left|\ \vphantom{#1} #2\ \right.\cbr}

\newcommand{\GL}{\text{GL}}

\newcommand{\cD}{\mathcal{D}}
\newcommand{\cK}{\mathcal{K}}
\newcommand{\cU}{\mathcal{U}}
\newcommand{\cF}{\mathcal{F}}
\newcommand{\cG}{\mathcal{G}}
\newcommand{\cN}{\mathcal{N}}
\newcommand{\cY}{\mathcal{Y}}

\DeclareMathOperator*{\rf}{ref}
\DeclareMathOperator*{\esssup}{ess\,sup}

\DeclareMathOperator*{\loc}{loc}

\newcommand{\e}{\varepsilon}

\renewcommand{\l}{\lambda}

\newcommand{\me}{\mathrm{e}}

\newcommand{\g}{\gamma}
\newcommand{\G}{\Gamma}
\newcommand{\s}{\sigma}

\newcommand{\con}{\mathcal{C}}

\RedeclareMathOperator*{\Im}{Im}
\RedeclareMathOperator*{\Re}{Re}
\renewcommand{\phi}{\varphi}

\renewcommand{\d}{\ \text{d}}

\newcommand{\dd}[2][ ]{\tfrac{\text{\normalfont d}#1}{\text{\normalfont d}#2}}

\DeclareMathOperator*{\im}{im}

\makeatletter
\@ifundefined{ifCustomTheorems}{}
{
    \setcounter{section}{0}
    \counterwithout{equation}{section}
    \counterwithout{figure}{section}
    \newtheorem{definition}{Definition}[section]
    \theoremstyle{definition}
    \newtheorem{remark}[definition]{Remark}
    \newtheorem{algo}[definition]{Algorithm}
    \newtheorem{example}[definition]{Example}
    \theoremstyle{plain}
    \newtheorem{prop}[definition]{Proposition}

    \newtheorem{theorem}[definition]{Theorem}
    \newtheorem{lemma}[definition]{Lemma}
}
\makeatother

\graphicspath{{resources/}}

\definecolor{sectioncolor}{RGB}{0, 0, 0}
\allowdisplaybreaks

\begin{document}
\newcommand{\FeasibControls}{\cU_{T}(t^0,x)}
\newcommand{\ControlFunctions}[4]{\cU^{#4}_{#1}(M,#2,#3)}

\newcommand{\Controls}{\ControlFunctions{T}{t^0}{x^0}{\phi}}
\newcommand{\PhiControls}[3]{\ControlFunctions{#1}{#2}{#3}{\phi}}
\newcommand{\PhiTControls}[2]{\PhiControls{T}{#1}{#2}}
\newcommand{\FunControls}[1]{\ControlFunctions{T}{t^0}{x^0}{#1}}

\newcommand\funding[1]{\protect{\bfseries Funding:} #1}
\newcommand{\email}[1]{\protect\href{mailto:#1}{#1}}
\newcommand\keywordsname{Key words}
\newcommand\AMSname{AMS subject classifications}
\newenvironment{@abssec}[1]{%
     \if@twocolumn
       \section*{#1}%
     \else
       \vspace{.05in}\footnotesize
       \parindent .2in
         {\upshape\bfseries #1: }\ignorespaces 
     \fi}
     {\if@twocolumn\else\par\vspace{.1in}\fi}
\newenvironment{AMS}{\begin{@abssec}{\AMSname}}{\end{@abssec}}
\newenvironment{keywords}{\begin{@abssec}{\keywordsname}}{\end{@abssec}}

\title{Funnel MPC for nonlinear systems with relative degree one
\thanks{
\funding{D. Dennstädt gratefully thanks the Technische Universität Ilmenau and the Free State of Thuringia
for their financial support as part of the Thüringer Graduiertenförderung.
T. Berger and K. Worthmann acknowledge funding by the German Research Foundation (DFG; grants
WO~2056/12-1, BE~6263/3-1, project number 471539468).
K. Worthmann gratefully acknowledges funding by the German Research Foundation (DFG; grant WO~2056/6-1, project number 406141926). 
}}}
\author{Thomas Berger
        \thanks{Institut für Mathematik, Universit\"{a}t Paderborn, Warburger Stra\ss e~100, 33098~Paderborn, Germany
        (\email{thomas.berger@math.upb.de}). }
    \and Dario Dennstädt
        \thanks{Institut für Mathematik, Technische Universität Ilmenau, Weimarer Stra\ss e 25, 98693~Ilmenau, Germany
    (\email{dario.dennstaedt@tu-ilmenau.de},
     \email{achim.ilchmann@tu-ilmenau.de},
     \email{karl.worthmann@tu-ilmenau.de}).}
    \and Achim Ilchmann\footnotemark[3]
    \and Karl Worthmann\footnotemark[3]
}

\date{\today}
\maketitle


\begin{abstract}
\noindent
    We show that Funnel MPC, a novel Model Predictive Control (MPC) scheme, allows tracking of
    smooth reference signals with prescribed performance for nonlinear multi-input multi-output
    systems of relative degree one with stable internal dynamics. The optimal control problem solved
    in each iteration of Funnel MPC resembles the basic idea of penalty methods used in
    optimization. To this end, we present a new stage cost design to mimic the high-gain idea of
    (adaptive) funnel control. We rigorously show initial and recursive feasibility of Funnel MPC
    without imposing terminal conditions or other requirements like a sufficiently long prediction
    horizon.
\end{abstract}





\section{Introduction}
Model Predictive Control~(MPC) is a well-established control technique which relies on the iterative
solution of Optimal Control Problems (OCPs), see the textbooks~\cite{grune2017nonlinear,
rawlings2017model}. Thanks to its applicability to multi-input multi-output nonlinear systems and
its ability to take control and state constraints directly into account, it is nowadays widely used
and has seen various applications; see e.g.~\cite{QinBadg03}.

A key property for applying MPC is recursive feasibility, meaning that solvability of the OCP at a
particular time instant automatically implies solvability of the OCP at the successor time instant.
Often, suitably designed terminal conditions (costs and constraints) are incorporated in the iteratively solved OCP 
to ensure recursive feasibility, see e.g.~\cite{rawlings2017model}
and the references therein. However, such (artificially introduced) terminal conditions complicate
the task of finding an initially-feasible solution by imposing additional state constraints. As a
consequence, the domain of the MPC feedback controller might become significantly smaller. An
alternative approach, which is based on so-called cost controllability~\cite{CoroGrun20}, is using a
sufficiently-long prediction horizon, see e.g. \cite{boccia2014stability} and the references
therein or \cite{EsteWort20} for an extension to continuous-time systems. It is worth to be noted
that both techniques become significantly more involved in the presence of time-varying state (or
output) constraints.

To overcome the outlined 
restrictions for a large system class, Funnel MPC (FMPC) was proposed
in~\cite{berger2019learningbased}, which allows for reference tracking such that the tracking error
evolves in a pre-specified, potentially time-varying performance funnel. To this end, output
constraints were incorporated in the~OCP. Then, both initial and recursive feasibility were
rigorously shown by using properties of the system class in consideration --~without imposing
additional terminal conditions and independent of the length of the prediction horizon. Moreover,
the range of applied control values and the overall performance were further improved by using a
``funnel-like'' stage cost, which penalises the tracking error and becomes infinite when approaching
the funnel boundary. 

In the present paper, we show that such funnel-inspired stage costs (slightly modified compared to its
predecessor proposed in~\cite{berger2019learningbased}) automatically ensure initial and recursive
feasibility for a class of nonlinear systems with relative degree one and, in a certain
sense, input-to-state stable internal dynamics without adding the (artificial) output constraints
used in~\cite{berger2019learningbased}. 
To this end, novel
optimization-based arguments are employed, which somehow resemble ideas underlying penalty methods.
We are convinced that, in principle, similar techniques may be used to extend the
presented analysis to systems with higher relative degree. This conjecture is substantiated by
numerical simulations, for which FMPC shows superior performance compared to both MPC with quadratic
stage cost and funnel control.

The novel stage cost used in FMPC is inspired by funnel control. The latter is a model-free
output-error feedback of high-gain type introduced in~2002 by~\cite{IlchRyan02b}, see also the
recent work~\cite{BergIlch21} for a comprehensive literature overview. The funnel controller is
adaptive, inherently robust and allows reference tracking for a fairly large class of systems solely
invoking structural assumptions, i.e.\ stable internal dynamics, known relative degree with a
sign-definite high-frequency gain matrix. Most importantly, tracking is achieved within a prescribed
funnel, that means a prescribed transient behaviour is guaranteed. The funnel controller proved
useful for tracking problems in various applications such as temperature control of chemical reactor
models~\cite{IlchTren04}, control of industrial servo-systems~\cite{Hack17}, underactuated multibody
systems~\cite{BergOtto19} and DC-link power flow control~\cite{SenfPaug14}. Since funnel control,
contrary to MPC, does not use a model of the system, the controller only reacts on the current
system state and cannot ``plan ahead''. This often results in high control values and a rapidly
changing control signal with peaks. Furthermore, the controller requires a high sampling rate to
stay feasible, see e.g.~\cite{berger2019learningbased}. In applications, this results in quite
demanding hardware requirements.

Instead of guaranteeing that the output signal always evolves within predefined boundaries, previous
results for reference tracking with MPC mostly focus on ensuring asymptotic stability of the
tracking error, see e.g.~\cite{aydiner2016periodic,kohler2019nonlinear}. These approaches usually
modify the optimization problem by adding terminal constraints. In~\cite{aydiner2016periodic}
and~\cite{kohler2019nonlinear} asymptotic stability of the tracking error is guaranteed by designing
terminal sets and terminal costs around a specific reference signal. A tracking MPC scheme without
such constraints is studied in~\cite{kohler2018nonlinear}. The theoretical results for this scheme
rely on utilizing a sufficiently long prediction horizon instead.
In order to ensure reference tracking in the presence of disturbances or dynamic uncertainties,
tube-based robust MPC schemes use tubes around the reference signal which always confine the actual
system output, see e.g.~\cite{mayne2005robust, limon2005robust, falugi2013getting, kohler2020computationally}.
These tubes encompass the uncertainties of the system and can usually not be arbitrarily chosen a
priori. By adding terminal costs, terminal sets and constraints to the optimization problem it is
ensured that the system output always evolves within these tubes.
In~\cite{yu2013tube,singh2017robust} complex nonlinear incremental Lyapunov functions and a
corresponding incrementally stabilizing feedback is calculated offline in order to ensure that the
control objective is satisfied.
For linear systems the tracking of a reference signal within constant bounds is studied
in~\cite{di2015reference}. This procedure relies on the calculation of robust control
invariant~(RCI) sets in order to ensure that state, input and performance constraints are met. An
extension of this approach which also accounts for external disturbances can be found
in~\cite{yuan2019bounded}. These RCI~sets, however, are not trivial to calculate for a given system
and the algorithm proposed in~\cite{di2015reference} may in general not terminate in finite time.
Barrier function based MPC~(see e.g.~\cite{WILLS20041415}) follows a similar idea as FMPC. This
approach also uses, as part of the cost function, a term which diverges to infinity for states
converging to the boundary of a given set. However, utilizing terminal conditions (costs and
constraints) remains necessary in order to ensure recursive feasibility and that constraints are
met. By using a different kind of cost function, FMPC can circumvent this disadvantage.

By combining ideas from funnel control with MPC, the resulting Funnel MPC allows tracking of
sufficiently smooth reference signals for nonlinear multi-input multi-output systems of relative
degree one within a prescribed performance funnel. FMPC circumvents the shortcomings of both
approaches and enables us to benefit from the best of both worlds: guaranteed feasibility (funnel
control), a (slightly) enlarged system class (regularity of the high gain matrix is sufficient), and
superior performance (MPC).


The present paper is organized as follows.
We start by formulating the considered control problem and the MPC~algorithm in
Section~\ref{Sec:Problem}. After presenting the considered system class and detailing our structural
assumptions, we present the main result of this paper. By using a ``funnel-like'' stage cost function, it
is possible to track a reference signal within a prescribed funnel with MPC and guarantee initial
and recursive feasibility for any prediction horizon and without any terminal or output
constraints. After presenting simulations and promising preliminary results of numerical experiments
on an extension of FMPC in Section~\ref{Sec:Simulations}, we carry out the proof of the main result
over several steps in Section~\ref{Sec:ProofMainResult}. Finally, conclusions are drawn in
Section~\ref{Sec:Conclusion}.\\

\noindent
\textbf{Notation:}
$\N$ and $\R$ denote natural and real numbers, respectively.
$\N_0:=\N\cup\{0\}$ and $\Rp:=[0,\infty)$.
$\Norm{\cdot}$~denotes a norm in $\R^n$. $\Norm{A}$ denotes the induced operator norm
$\Norm{A}:=\sup_{\Norm{x} = 1}\Norm{Ax}$ for $A\in\R^{n\times m}$.
$\GL_n(\R)$ is the group of invertible $\R^{n\times n}$ matrices.
$\con^p(V,\R^n)$ is the linear space of $p$-times continuously  differentiable
functions $f:V\to\R^n$, where $V\subset\R^m$ and $p\in\N_0\cup \{\infty\}$.
$\con(V,\R^n):=\con^0(V,\R^n)$.
On an interval $I\subset\R$,  $L^\infty(I,\R^n)$ denotes the space of measurable essentially bounded
functions $f: I\to\R^n$ with norm $\SNorm{f}:=\esssup_{t\in I}\Norm{f(t)}$,
$L^\infty_{\text{loc}}(I,\R^n)$ the space of locally bounded measurable functions, and $L^p(I,\R^n)$
the space of measurable $p$-integrable functions with norm $\LNorm[p]{\cdot}$ and with $p\in\N$.
Further, $W^{k,\infty}(I,\R^n)$ is the Sobolev space of all $k$-times weakly differentiable functions
$f:I\to\R^n$ such that $f,\dots, f^{(k)}\in L^{\infty}(I,\R^n)$.\\




\section{Problem formulation and structural assumptions}\label{Sec:Problem}

We consider control affine multi-input multi-output systems
\begin{equation}\label{eq:Sys}
    \begin{aligned}
        \dot{x}(t)  & = f(x(t)) + g(x(t)) u(t),\quad x(t^0)=x^0,\\
        y(t)        & = h(x(t)),
    \end{aligned}
\end{equation}
with~$t^0\in\Rp$, $x^0\in\R^n$, functions~$f\in \con^1(\R^n,\R^n)$, $g\in \con^1(\R^n,\R^{n\times
m})$, $h \in \con^2(\R^n,\R^m)$, and control input function $u \in L^\infty_{\loc}(\R_{\geq 0},
\R^m)$. Note that both output~$y$ and input~$u$ have the same dimension. Due to the fact that the
input~$u$ does not have to be continuous, we use the generalised notion of \textit{Carath\'{e}odory
solutions} for ordinary differential equations, i.e., a function $x:[t^0,\omega)\to\R^n$,
$\omega>t^0$, with $x(t^0)=x^0$ is a solution of~\eqref{eq:Sys}, if it is absolutely continuous and
satisfies the ODE in~\eqref{eq:Sys} for almost all $t\in [t^0,\omega)$. A (Carath\'{e}odory)
solution $x:[t^0,\omega)\to\R^n$ is global, if $\omega=\infty$ and $x$ is a solution
of~\eqref{eq:Sys} on $[t^0,T)$ for all $T>t^0$. A solution $x$ is said to be maximal, if it has no
right extension that is also a solution. Any maximal solution of~\eqref{eq:Sys} is called the
\textit{response} associated with $u$ and denoted by~$x(\cdot;t^0,x^0,u)$. The response is unique
since the right-hand side of~\eqref{eq:Sys} is locally Lipschitz in~$x$, cf.~\cite[\S~10,
Thm.~XX]{Walt98}.


\subsection{Control objective}

Our objective is to design a control strategy which allows reference tracking of a given reference
trajectory~$y_{\rf}\in W^{1,\infty}(\Rp,\R^{m})$ within pre-specified error bounds. To be more
precise, the tracking error ~$t\mapsto e(t):=y(t)-y_{\rf}(t)$ shall evolve within the prescribed
performance funnel
\begin{align*}
    \cF_\phi:= \setdef{(t,e)\in \Rp\times\R^{m}}{\phi(t)\Norm{e} < 1}.
\end{align*}
This funnel is determined by the choice of the function~$\phi$
belonging  to
\begin{align*}
    \cG:=\setdef
        {\phi\in W^{1,\infty}(\Rp,\R)}
        {
            \inf_{t\geq 0}\phi(t) > 0
        \!\!},
\end{align*}
see also Figure~\ref{Fig:funnel}.
 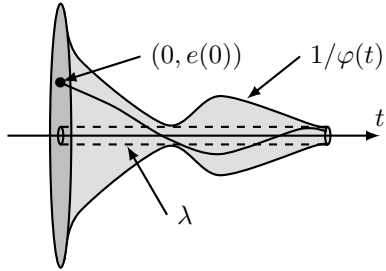
\begin{figure}[h]
  \begin{center}
\begin{tikzpicture}[scale=0.35]
\tikzset{>=latex}
  \filldraw[color=gray!25] plot[smooth] coordinates {(0.15,4.7)(0.7,2.9)(4,0.4)(6,1.5)(9.5,0.4)(10,0.333)(10.01,0.331)(10.041,0.3) (10.041,-0.3)(10.01,-0.331)(10,-0.333)(9.5,-0.4)(6,-1.5)(4,-0.4)(0.7,-2.9)(0.15,-4.7)};
  \draw[thick] plot[smooth] coordinates {(0.15,4.7)(0.7,2.9)(4,0.4)(6,1.5)(9.5,0.4)(10,0.333)(10.01,0.331)(10.041,0.3)};
  \draw[thick] plot[smooth] coordinates {(10.041,-0.3)(10.01,-0.331)(10,-0.333)(9.5,-0.4)(6,-1.5)(4,-0.4)(0.7,-2.9)(0.15,-4.7)};
  \draw[thick,fill=lightgray] (0,0) ellipse (0.4 and 5);
  \draw[thick] (0,0) ellipse (0.1 and 0.333);
  \draw[thick,fill=gray!25] (10.041,0) ellipse (0.1 and 0.333);
  \draw[thick] plot[smooth] coordinates {(0,2)(2,1.1)(4,-0.1)(6,-0.7)(9,0.25)(10,0.15)};
  \draw[thick,->] (-2,0)--(12,0) node[right,above]{\normalsize$t$};
  \draw[thick,dashed](0,0.333)--(10,0.333);
  \draw[thick,dashed](0,-0.333)--(10,-0.333);
  \node [black] at (0,2) {\textbullet};
  \draw[->,thick](4,-3)node[right]{\normalsize$\lambda$}--(2.5,-0.4);
  \draw[->,thick](3,3)node[right]{\normalsize$(0,e(0))$}--(0.07,2.07);
  \draw[->,thick](9,3)node[right]{\normalsize$1/\varphi(t)$}--(7,1.4);
\end{tikzpicture}
\end{center}
 \vspace*{-2mm}
 \caption{Error evolution in a funnel $\mathcal F_{\varphi}$ with boundary $1/\varphi(t)$.}
 \label{Fig:funnel}
 \end{figure}

Note that boundedness of $\varphi$ implies that there exists $\lambda>0$ such that
$1/\varphi(t)\geq\lambda$ for all $t \ge 0$. Therefore, signals evolving in $\mathcal{F}_{\varphi}$
are not forced to converge to $0$ asymptotically. To achieve that  the tracking error~$e$ remains
within~$\cF_\phi$, it is necessary that the solution~$x$ of the system~\eqref{eq:Sys} evolves within
the set
\[
    \cD^{\phi} := \setdef
                    { (t,x)\in \Rp \times \R^n}
                    {\phi(t) \Norm{h(x) - y_{\rf}(t)}<1}.
\]
To simplify notation we denote by~$\cD^{\phi}_t$ the second component of the set~$\cD^{\phi}$ at
time~$t\in\Rp$, meaning
\begin{align}\label{eq:Def-Dt}
    \cD^{\phi}_t := \setdef
                    {x\in\R^n}
                    {\phi(t) \Norm{h(x) - y_{\rf}(t)} <1}.
\end{align}
\begin{remark}
    In many practical applications perfect tracking is neither possible nor desired. Usually, the
    objective rather is to ensure the tracking error to be less than an (arbitrary small)
    prespecified constant after a prespecified period of time and to guarantee that the error does
    not exceed this bound at a later time.
    Tracking within a funnel, or in other words practical tracking, is advantageous since it allows
    tracking for system classes where asymptotic tracking is not possible or requires – when compared to
    asymptotic tracking – much less control effort.
    Note  that the function~$\varphi$ is a design parameter, thus its choice is completely up to the
    designer.
    Moreover, arbitrary funnel functions -- and not restricted to
    constant or monotonous decreasing funnels -- give the user more flexibility in
    finding a suitable trade-off between tracking performance and control effort. Typically, the
    specific application dictates the constraints on the tracking error and thus indicates suitable
    choices for $\varphi$. During safety critical system phases, the funnel will be small, while
    during non-critical phases the funnel can be widened again to reduce the control effort.
\end{remark}



\subsection{MPC with quadratic stage cost}

\noindent
The idea of Model Predictive Control~(MPC) is, after measuring/obtaining the state
$x(\widehat{t}) = \widehat{x} \in \R^{n}$ ($\widehat{t} \geq t^0$) at the current time~$\widehat{t}$, to repeatedly
calculate a control function~$u^{\star} = u^{\star}(\cdot;\widehat{t},\widehat{x})$  minimizing the
integral of a state cost~$\ell$ on the time interval $[\widehat{t},\widehat{t}+T]$ for $T>0$ and
implement the computed optimal solution~$u^{\star}$ to system~\eqref{eq:Sys} over an interval of
length $\delta<T$. $T$ and $\delta$ are called the prediction horizon and time shift, respectively.
It is clear that necessarily the solution ~$x(\cdot;\widehat{t},\widehat{x},u)$
of the system~\eqref{eq:Sys} exists on the whole interval~$[\widehat{t},\widehat{t}+T]$, i.e.,
$u^{\star}$ has to be an element of the set
\begin{equation*}\label{eq:DefFeasControls}
    \cU_{T}(\widehat{t},\widehat{x}) := \setdef
                      { u\in L^\infty([\widehat{t},\widehat{t}+T],\R^m) }
                      { x(t;\widehat{t},\widehat{x},u)\text{ satisfies~\eqref{eq:Sys} for all } t\in [\widehat{t},\widehat{t}+T] }.
\end{equation*}
When solving the problem of tracking a reference signal $y_{\rf}$, the \textit{stage cost}
\begin{equation}\label{eq:stageCostClassicalMPC}
    \begin{aligned}
        \ell:\Rp\times\R^n\times\R^{m}&\to\R,\qquad
        (t,x,u) \mapsto              \Norm{h(x)-y_{\rf}(t)}^2+\l_u \Norm{u}^2
    \end{aligned}
\end{equation}
with $\l_u>0$ is usually used. While the term $\Norm{h(x)-y_{\rf}(t)}^2$ penalises the distance of
the output $y=h(x)$ to the reference signal $y_{\rf}$, the term $\Norm{u}^2$ penalises the control
effort. The parameter~$\l_u$ allows to adjust a suitable trade-off between tracking performance and
required control effort. Of course, if a reference input signal~$u_{\rf}$ is known, the second
summand may be replaced by $\| u - u_{\rf}(t) \|^2$. To guarantee that the tracking error $e$
evolves within the prescribed funnel one adds the additional constraint
\begin{equation}\label{eq:ConstraintClassicalMPC}
    \fa t\in[\widehat{t},\widehat{t}+T]:\quad \phi(t)\Norm{y(t)-y_{\rf}(t)} \leq 1
\end{equation}
to the optimization problem, cp.~\cite{berger2019learningbased}. To ensure a bounded control signal,
one additionally adds the constraint $\Norm{u(t)}  \leq M$ for a predefined constant $M>0$.\\

\begin{algo}[MPC]\label{Algo:MPC}\ \\
    \textbf{Given:} System~\eqref{eq:Sys}, reference signal $y_{\rf}\in
    W^{1,\infty}(\Rp,\R^{m})$, funnel function $\phi\in\cG$, $M>0$, $t^0\in\Rp$,
    $x^0\in\cD^{\phi}_{t^0}$, and stage cost function~$\ell$ as in~\eqref{eq:stageCostClassicalMPC}.\\
    \textbf{Set} the time shift $\delta >0$, the prediction horizon $T\geq\delta$, and 
    the current time
        $\widehat t :=t^0$.\\
    \textbf{Steps:}
    \begin{enumerate}[(a)]
    \item\label{agostep:MPCFirst} Obtain a measurement of the state 
    at time~$\widehat t$ and set $\widehat x :=x(\widehat t)$.
    \item Compute a solution $u^{\star}\in L^\infty([\widehat t,\widehat t +T],\R^{m})$ of
    \begin{equation}\label{eq:MpcOCP}
    \begin{alignedat}{2}
            &\!\mathop
            {\operatorname{minimize}}_{u\in L^{\infty}([\widehat t,\widehat t+T],\R^{m})}  &\quad&
            \int_{\widehat t}^{\widehat t+T}\ell(t,x(t;\hat{t},\hat{x},u),u(t))\d t \\
            &\text{subject to} &       x(t;\hat{t},\hat{x},u)&\in\cD_t^\varphi,\\
            &                  &       \Norm{u(t)}  &\leq M.
    \end{alignedat}
    \end{equation}
    \item Apply the feedback law
        \[
            \mu:[\widehat t,\widehat t+\delta)\times\R^n\to\R^m, \quad \mu(t,\widehat x) =u^{\star}(t)
        \]
        to system~\eqref{eq:Sys}.
        Increase $\widehat t$ by $\delta$ and go to Step~\ref{agostep:MPCFirst}.
    \end{enumerate}
\end{algo}



\subsection{Drawbacks of the MPC scheme~\ref{Algo:MPC}}\label{Sec:DrawbackClassicMPCScheme}

\noindent Although utilizing the stage cost~$\ell$ in~\eqref{eq:stageCostClassicalMPC} and
constraints~\eqref{eq:ConstraintClassicalMPC} in Algorithm~\ref{Algo:MPC} might seem like a canonical
choice when solving the reference tracking problem with MPC, this approach has several drawbacks.
In particular, one has to guarantee initial and recursive feasibility of the MPC
Algorithm~\ref{Algo:MPC}. This means, it is necessary to prove that the optimization
problem~\eqref{eq:MpcOCP} has initially (i.e., at $t=t^0$) and recursively (i.e., at $t=t^0+\delta
n$ after $n$ steps of Algorithm~\ref{Algo:MPC}) a solution. First of all, one has to
show existence of an $L^\infty$-control $u$ bounded by $M>0$ which, if applied to a restricted
system class of~\eqref{eq:Sys}, guarantees that the tracking error~$e(t)=y(t)-y_{\rf}(t)$ evolves within the
performance funnel, i.e.,
\[
    \fa t\in[t^0,t^0+T]:\ \phi(t)\Norm{e(t)} = \phi(t)\Norm{y(t)-y_{\rf}(t)} < 1.
\]
Or, formulating it slightly different, one has to show that for $t^0\in\Rp$, $x^0\in\R^n$, $M>0$,
and $T>0$ the set
\begin{equation}\label{eq:DefUL2}
    \Controls{} := \setdef
                                {u\in \FeasibControls}
                                { \!\!
                                        \fa t\in[t^0,t^0+T]:\ x(t;t^0,x^0,u)\in\cD^{\phi}_t,\
                                        \SNorm{u}<M
                                \!\!\!}
\end{equation}
is non-empty. Note that, for $\Controls\neq\emptyset$, it is necessary that the initial error is
contained in the interior of the funnel, i.e., $x^0\in\cD^{\phi}_{t^0}$. Furthermore, one has to
show that there exists a solution~$u^{\star}$ of the optimization problem~\eqref{eq:MpcOCP} and this
solution is an element of~$\Controls{}$.

To show recursive feasibility, it is further necessary to prove that after applying a
solution~$u^{\star}$ of the optimal control problem~\eqref{eq:MpcOCP} at time $t=t^0+\delta n$ to
the system~\eqref{eq:Sys} the optimization problem is still well defined at the next time instant
$\hat{t}=t^0+\delta (n+1)$, i.e., the set~$\PhiTControls{\hat{t}}{\hat x}$ is non-empty, where $\hat
x$ is the state of the system at time $\hat{t}$. To guarantee this recursive feasibility of the MPC
scheme in consideration, a sufficiently long prediction horizon $T$ (see
e.\,g.~\cite{boccia2014stability}) or suitable terminal constraints (see
e.g.~\cite{rawlings2017model}) are usually required while initial feasibility (i.e.
$\Controls\neq\emptyset$) is assumed. Moreover, the time-varying (state/output)
constraints~\eqref{eq:ConstraintClassicalMPC} in the optimization problem~\eqref{eq:MpcOCP} pose an
additional challenge; both for the theoretical analysis and also from a numerical point of view.
\begin{remark}
 Note that for two functions $\phi,\psi\in\cG$ with $\psi(t)\geq\phi(t)$ for all $t\in[t^0,t^0+T]$, we have
    \[
        \FunControls{\psi}\subseteq \FunControls{\phi}.
    \]
\end{remark}

Before we show how to overcome these drawbacks by a new stage cost in
Section~\ref{Sec:NewStageCost}, we introduce the class of systems to which our approach is restricted.



\subsection{System class}
Throughout this work we assume that system~\eqref{eq:Sys} has known relative degree $r=1$, i.e.,
the \emph{high-frequency gain matrix}
\begin{equation}
    \G(x) := \rbl h'g\rbr(x) \in \GL_m(\R)\quad \fa x\in\R^n. \label{eq:Gamma}
\end{equation}
Additionally, we assume that $h^{-1}(0)$ is diffeomorphic to $\R^{n-m}$ and the
distribution\footnote{By a distribution, we mean a mapping from $\R^n$ to the set of all subspaces
of $\R^n$.} $x\mapsto \cG(x) := \im g(x)$ is involutive, i.e., for all smooth vector fields
$\psi_1,\psi_2:\R^n\to\R^n$ with $\psi_i(x)\in\cG(x)$ for all $x\in\R^n$ and $i=1,2$ we have that
the Lie bracket $[\psi_1, \psi_2](x) = \psi_1'(x) \psi_2(x) - \psi_2'(x)\psi_1(x)$ satisfies
$[\psi_1, \psi_2](x)\in\cG(x)$ for all $x\in\R^n$. Note that for single-input, single-output systems
(i.e., $m=1$) the distribution $\cG(x)$ is always involutive.
Then, by~\cite[Cor. 5.7]{ByrnIsid91a} there exists a diffeomorphism~$\Phi:\R^n\to\R^n$ such that
the coordinate transformation $(y(t),\eta(t)) = \Phi(x(t))$ puts the system into Byrnes-Isidori form
\begin{subequations}\label{eq:BIF}
    \begin{align}
        &\dot y(t) = p\rbl y(t),\eta(t)\rbr + \G\rbl \Phi^{-1}\rbl y(t),\eta(t)\rbr\rbr\,u(t),\quad\ (y(t^0),\eta(t^0)) = (y^0,\eta^0) = \Phi(x^0),
        \label{eq:output_dyn}\\
        &\dot \eta(t) = q\rbl y(t),\eta(t)\rbr,\label{eq:zero_dyn}
    \end{align}
\end{subequations}
where $p\in \con^1(\R^m\times\R^{n-m},\R^m)$ and $q\in \con^1(\R^m\times\R^{n-m},\R^{n-m})$.
Furthermore, we impose the following version of a \emph{bounded-input, bounded-state} (BIBS) condition on the
internal dynamics~\eqref{eq:zero_dyn}:
\begin{multline}\label{eq:BIBO-ID}
        \fa c_0 >0  \ex c_1 >0  \fa t^0\in\Rp \fa  \eta^0\in\R^{n-m}
        \fa  y\in L^\infty_{\loc}([t^0,\infty),\R^m):\\[2ex]
        \Norm{\eta^0}+
        \SNorm{y}  \leq c_0\ \Impl\ \SNorm{\eta (\cdot;t^0,\eta^0,y)} \leq c_1,
\end{multline}
where (here and throughout the paper)  $\eta (\cdot;t^0, \eta^0,y):[t^0,\infty)\to\R^{n-m}$ denotes
the unique global solution of~\eqref{eq:zero_dyn}. Here, the maximal solution $\eta (\cdot;t^0,
\eta^0,y)$ of~\eqref{eq:zero_dyn} can indeed be extended to a global solution since it is bounded by
the BIBS condition~\eqref{eq:BIBO-ID}, cf.~\cite[\S~10, Thm.~XX]{Walt98}.

\begin{remark}
    If a stabilizing state feedback $u=Fx$ is applied to a system of the form 
    \begin{subequations}
    \begin{align}
        \dot{x}(t)&=Ax(t)+Bu(t),\label{eq:SusmannPeakingLinearPart}\\
        \dot{\eta}(t)&=f(x(t),\eta(t))\label{eq:SusmannPeakingNonLinearPart}
    \end{align}
    \end{subequations}
    with controllable $(A,B)\in\mathds{R}^{n\times n}\times\mathds{R}^{n\times m}$ and
    continuously differentiable function $f:\mathds{R}^n\times\mathds{R}^{\ell}\to\mathds{R}^{\ell}$, then the
    linear part~\eqref{eq:SusmannPeakingLinearPart} can be estimated, for $t\geq0$ and $a,\kappa>0$, by $\|x(t)\|\leq\kappa
    e^{-at}\|x(0)\|$. Any prespecified~$a$ can be realized by the choice of~$F$.
    However, as stated by Sussmann and Kokotovic in~\cite{sussmann1991peaking}, \textit{one
    cannot, in general, choose $F$ so as to make the number $a$ large without making~$\kappa$ large as
    well}. As first pointed out by Sussmann in~\cite{sussmann1990limitations}, the so called
    \textit{peaking-phenomenon} can cause the nonlinear part~\eqref{eq:SusmannPeakingNonLinearPart} of
    the system to have finite escape time even if the system
    \[
        \dot{\eta}(t)=f(0,\eta(t))
    \]
    has $0$ as a global asymptotically stable equilibrium.
    The presumed BIBS condition~\eqref{eq:BIBO-ID} not only avoids this problem, but is even more
    essential since our control objective is to guarantee that the system output~$y$ evolves within
    the funnel around the reference signal~$y_{\text{ref}}$. Without this assumption and even with
    perfect tracking, the internal dynamics might~\eqref{eq:zero_dyn} be unbounded and thus cause an
    unbounded control effort, or worse, its solution might even have finite escape time.
\end{remark}

We summarize our assumptions and define the general system class to be considered.
\begin{definition}[System class]
    We say that the system~\eqref{eq:Sys} belongs to the \emph{system class}~$\cN^{m}$,
    written~$(f,g,h)\in\cN^{m}$, if it has global relative degree~$r=1$, $h^{-1}(0)$ is isomorphic
    to $\R^{n-m}$, the distribution $x\mapsto
    \im g(x)$ is involutive, and the system satisfies the~BIBS condition~\eqref{eq:BIBO-ID}.
\end{definition}

\begin{remark}
    Relaxing standard requirements in funnel-control (see e.g.~\cite{IlchRyan02b,BergIlch21}), the high-frequency
    gain matrix $\G(x)$ does not need to be sign-definite.
    We only require the much weaker assumption of~$\G(x)$ being invertible.
\end{remark}

We further emphasize that these structural assumptions are sufficient conditions for our results, but
they are not necessary. First promising preliminary simulation results show that Funnel MPC can also
successfully be applied to a more general system class (see Section~\ref{SubSec:MassOnCar}).



\subsection{Novel stage cost design}\label{Sec:NewStageCost}

\noindent
To overcome the drawbacks of the MPC scheme~\ref{Algo:MPC} outlined in Section~\ref{Sec:DrawbackClassicMPCScheme}, we propose for $\phi\in\cG$, $y_{\rf}\in W^{1,\infty}(\Rp,\R^{m})$, and  design parameter
${\l_u\in\Rp}$ the new \textit{stage cost function}
\begin{equation}\label{eq:stageCostFunnelMPC}
    \begin{aligned}
        \ell_{\phi}:\Rp\times\R^n\times\R^{m}&\to\R\cup\{\infty\},\\
        (t,x,u) &\mapsto
        \begin{dcases}
            \frac {1}{1-\phi(t)^2 \Norm{h(x)-y_{\rf}(t)}^2} -1 + \l_u \Norm{u}^2,
                & \phi(t)\Norm{h(x)-y_{\rf}(t)} \neq 1\\
            \infty,&\text{else},
        \end{dcases}
    \end{aligned}
\end{equation}
to be used in the MPC Algorithm~\ref{Algo:MPC} instead of $\ell$
from~\eqref{eq:stageCostClassicalMPC}. The term~$\frac {1}{1-\phi(t)^2 \Norm{h(x)-y_{\rf}(t)}^2}$
penalises the distance of the tracking error to the funnel boundary, whereas the parameter $\l_u$
again influences the penalization of the control input. Note that we allow for $\l_u=0$.

The cost function~$\ell_{\phi}$ is motivated by the following standard result on funnel control
from~\cite[Thm.~7]{IlchRyan02b}.

\begin{prop}\label{Prop:FC}
    Assume that $(f,g,h)\in\cN^{m}$, $\phi\in\cG$, $y_{\rf}\in W^{1,\infty}(\Rp,\R^{m})$,
    $t^0\in\Rp$, and $x^0\in\cD^{\phi}_{t^0}$. Further assume that the high-frequency gain matrix
    $\Gamma(x)$ as in~\eqref{eq:Gamma} is positive definite for all $x\in\R^n$. Then the application
    of the output feedback $u(t):=\mu_{\rm FC}(t,y(t))$ with
  \begin{equation}\label{eq:FC}
      \mu_{\rm FC}(t,y) = -k(t,y) e(t,y),\qquad
      k(t,y)          = \frac{1}{1-\phi(t)^2 \Norm{e(t,y)}^2},\qquad
      e(t,y)          = y - y_{\rf}(t)
  \end{equation}
  to~\eqref{eq:Sys} leads to the closed-loop initial value problem
  \begin{align*}
      \dot{x}(t) & = f(x(t)) - g(x(t))
                \frac{y(t) - y_{\rf}(t)}{1-\phi(t)^2\Norm{y(t) - y_{\rf}(t)}^2},\quad x(t^0)=x^0,\\
      y(t) & = h(x(t)),
  \end{align*}
    which has a solution,  every
    solution can be extended to a unique global solution $x:[t^0,\infty)\to\R^n$, and $x, u, y$ are bounded
    with essentially bounded weak derivatives. The tracking error evolves uniformly within the
    performance funnel, i.e.,
    \[
        \ex \e>0\fa t>0:\ \Norm{e(t)} \le \phi(t)^{-1} - \e.
    \]
\end{prop}

\begin{remark}\label{Rem:FunnelBound}
The following holds according to Proposition~\ref{Prop:FC}:
\[
    \fa x^0\in\cD^{\phi}_{t^0}\ex M>0:\quad \Controls \neq \emptyset.
\]
Note that in general the bound $M$ depends on $f,g,h,\phi$, and $x^0$.
\end{remark}



\subsection{Main result}\label{Sec:MainResult}

We are now in the position to define the Funnel MPC (FMPC) algorithm. It is the
MPC~Algorithm~\ref{Algo:MPC} without the output constraint~\eqref{eq:ConstraintClassicalMPC} and
cost function~$\ell$ as in~\eqref{eq:stageCostClassicalMPC} replaced by~$\ell_\phi$ as
in~\eqref{eq:stageCostFunnelMPC}.
\begin{algo}[FMPC]\label{Algo:MPCFunnelCost}\ \\
    \textbf{Given:} System~\eqref{eq:Sys},  reference signal $y_{\rf}\in
    W^{1,\infty}(\Rp,\R^{m})$, funnel function $\phi\in\cG$, $M>0$, $t^0\in\Rp$,
    $x^0\in\cD^{\phi}_{t^0}$, and stage cost function~$\ell_{\phi}$ as in~\eqref{eq:stageCostFunnelMPC}.\\
    \textbf{Set} the time shift $\delta >0$, the prediction horizon $T\geq\delta$ and initialize the current time
        $\widehat t :=t^0$.\\
    \textbf{Steps:}
    \begin{enumerate}[(a)]
    \item\label{agostep:FunnelMPCFirst}
        Obtain a measurement of the state at~$\hat t$ and set $\widehat x :=x(\widehat t)$.
    \item
        Compute a solution $u^{\star}\in L^\infty([\widehat t,\widehat t +T],\R^{m})$ of the Optimal
        Control Problem (OCP)
    \begin{equation}\label{eq:FunnelMpcOCP}
        \mathop
                {\operatorname{minimize}}_{\substack
                {
                    u\in L^{\infty}([\widehat t,\widehat t+T],\R^{m}),\\
                    \SNorm{u}  \leq M
                }
            }\     \int_{\widehat t}^{\widehat t+T}\ell_\phi(t,x(t;\widehat t,\widehat x,u),u(t))\d t
    \end{equation}
    \item Apply the feedback law
        \begin{equation}\label{eq:FMPC-fb}
            \mu:[\widehat t,\widehat t+\delta)\times\R^n\to\R^m, \quad \mu(t,\widehat x) =u^{\star}(t)
        \end{equation}
        to system~\eqref{eq:Sys}.
        Increase $\widehat t$ by $\delta$ and go to Step~\ref{agostep:FunnelMPCFirst}.

    \end{enumerate}
\end{algo}

We show that the Funnel MPC Algorithm~\ref{Algo:MPCFunnelCost} is initially and recursively feasible for
every prediction horizon $T>0$. Application of FMPC to system~\eqref{eq:Sys} with
$(f,g,h)\in\cN^{m}$ guarantees tracking of a reference trajectory~$y_{\rf}\in
W^{1,\infty}(\Rp,\R^{m})$ within a prescribed performance funnel defined by~$\phi\in\cG$.

\begin{theorem}\label{Th:FunnelMPCRelDeg1}
    Consider system~\eqref{eq:Sys} with $(f,g,h)\in\cN^{m}$. Let $\phi\in\cG$, $\ y_{\rf}\in
    W^{1,\infty}(\Rp,\R^{m})$, $t^0\in\Rp$ and $B\subset\cD^{\phi}_{t^0}$ be a bounded set.
    Then there exists $M>0$ such that the FMPC
    Algorithm~\ref{Algo:MPCFunnelCost} with $T>0$ and $\delta>0$ is initially and
    recursively feasible for every $x^0\in B$, i.e., at time $\widehat t = t^0$ and at each
    successor time $\widehat t\in t^0+\delta\N$ the OCP~\eqref{eq:FunnelMpcOCP}
    has a solution. In particular, the closed-loop system consisting of~\eqref{eq:Sys} and the FMPC feedback~\eqref{eq:FMPC-fb} has a (not necessarily unique) global solution $x:[t^0,\infty)\to\R^n$ and the corresponding input is given by
    \[
        u_{\rm FMPC}(t) = \mu(t,x(\widehat t)),\quad t\in [\widehat t,\widehat t+\delta),\ \widehat t\in t^0+\delta\N.
    \]
    Furthermore, each global solution~$x$ with corresponding input $u_{\rm FMPC}$ satisfies:
    \begin{enumerate}[(i)]
        \item\label{th:item:BoundedInput}
$\fa t\ge t^0:\quad \Norm{u_{\rm FMPC}(t)}\leq M$.
        \item\label{th:item:ErrorInFunnel}The error $e=y-y_{\rf}$ evolves within the funnel
            $\cF_{\phi}$, i.e., $\Norm{e(t)} \le \phi(t)^{-1}$ for all $t\ge t^0$.
    \end{enumerate}
\end{theorem}

\begin{remark}
    \begin{enumerate}[(a)]
        \item
            The OCP~\eqref{eq:FunnelMpcOCP} has neither state nor terminal constraints.
            Nevertheless, application of the FMPC Algorithm~\ref{Algo:MPCFunnelCost} to the
            system~\eqref{eq:Sys} ensures that a global solution of the closed-loop system exists
            and the error evolves within the funnel. However, note that this solution is not unique
            in general. The reason is that the solution of the OCP~\eqref{eq:FunnelMpcOCP} found in
            each step may not be unique. The MPC algorithm has to select a particular optimal
            control. In particular, Theorem~\ref{Th:FunnelMPCRelDeg1} shows that the
            properties~\ref{th:item:BoundedInput} and~\ref{th:item:ErrorInFunnel} are independent
            of the particular choice made within the MPC algorithm, since they hold for every such solution.
        \item
            FMPC is initially and recursively feasible for every choice of $T>0$. Usually, recursive
            feasibility for Model Predictive Control can only be guaranteed when the prediction
            horizon is sufficiently long (see, e.g.~\cite{boccia2014stability}) or when additional
            terminal constraints are added to the OCP (see, e.g.~\cite{rawlings2017model}). For FMPC
            merely the input constraints given by $M>0$ must be sufficiently large.
    \end{enumerate}
\end{remark}

The proof is carried out over several steps in Section~\ref{Sec:ProofMainResult}. In
Section~\ref{SubSec:OCP} we first assume that the set~$\Controls{}$ is non-empty and prove that the
optimization problem \eqref{eq:FunnelMpcOCP} has a solution~$u^{\star}$ and this solution is an
element of $\Controls{}$. We further show that the stage cost function~$\ell_\phi$ as
in~\eqref{eq:stageCostFunnelMPC} guarantees that application of $u^{\star}$ ensures that the
tracking error $e=y-y_{\rf}$ evolves within the funnel~$\cF_{\phi}$. In Section~\ref{SubSec:Feasibility} we
prove initial and recursive feasibility of the Funnel MPC Algorithm~\ref{Algo:MPCFunnelCost} by
showing that there exists  $M>0$ such that the set $\PhiControls{T}{\widehat{t}}{\widehat{x}} $ is initially
(i.e., at $\widehat{t}=t^0$) and recursively (i.e., at $\widehat{t}=t^0+\delta n$ after $n$ steps of
Algorithm~\ref{Algo:MPCFunnelCost}) non-empty, where $\widehat{x}$ is the state of the system at
time~$\widehat{t}$.




\section{Examples/Simulations}\label{Sec:Simulations}
\begin{example}[Linear system]\label{Example:LinearSystem}
    To illustrate the system class~$\cN^{m}$, we consider the example of a linear
    time-invariant system of the form
    \begin{equation}\label{eq:LinearSys}
        \begin{aligned}
            \dot{x}(t)  & = A x(t) + B u(t),\quad x(t^0)=x^0\\
            y(t)        & = C x(t),
        \end{aligned}
    \end{equation}
    where $(A,B,C) \in \R^{n \times n} \times \R^{n \times m} \times \R^{m \times n}$. This linear
    system has global relative degree~$r=1$, if
    \[
            CB  \in \GL_m(\R).
    \]
    It is shown in~\cite[Lemma~2.1.3]{Ilch93} that there exists an invertible matrix~$V\in\R^n$ such
    that the coordinate transformation
    \[
        \Phi(x):=Vx=(y,\eta)
    \]
    transforms the system~\eqref{eq:LinearSys} into the Byrnes-Isidori form
    \begin{equation}\label{eq:LinearSysByrnesIsidori}
        \begin{aligned}
            \dot{y}(t)      & = A_1 y(t) + A_2 \eta(t)  + \G u(t),\quad (y(t^0), \eta(t^0))=\Phi(x^0)\\
            \dot{\eta}(t)   & = A_3 y(t) + A_4 \eta(t),
        \end{aligned}
    \end{equation}
    with $(A_1, A_2, A_3, A_4)\in
           \R^{m \times m} \times \R^{m \times (n-m) } \times
           \R^{(n-m) \times m} \times \R^{(n-m)\times (n-m)}$.
    It is well known from the theory of linear differential equations that, if~$A_4$ is Hurwitz,
    i.e.,  all of its eigenvalues have negative real part, then $\eta (\cdot;t^0,\eta^0,y)$ is
    bounded for every $y\in L^\infty(\R_{\ge 0},\R^m)$. The~BIBS condition~\eqref{eq:BIBO-ID} is
    therefore satisfied in this case.
\end{example}


\subsection{Exothermic chemical reaction}

To demonstrate the application of the FMPC Algorithm~\ref{Algo:MPCFunnelCost}, we consider a
model of an exothermic chemical reaction which was used in~\cite{IlchTren04} to study funnel control
with input saturation and in~\cite{LibeTren10} to demonstrate the feasibility of the bang-bang
funnel controller. The model for one reactant $x_1$, one product $x_2$ and temperature $y$ of the
reactor is given by the equations
\begin{equation}\label{eq:ExampleExothermicReaction}
    \begin{aligned}
        \dot{y}(t)   &= b\,   p(x_1(t), x_2(t), y(t)) -q\,y(t) + u(t),\\
        \dot{x}_1(t) &= c_1\, p(x_1(t), x_2(t), y(t)) +d(x_1^{\text{in}}-x_1(t)),\\
        \dot{x}_2(t) &= c_2\, p(x_1(t), x_2(t), y(t)) +d(x_2^{\text{in}}-x_2(t)),
    \end{aligned}
\end{equation}
with $b,d,q\in\Rpp$, $c_1<0$, $c_2\in\R$,  $x_{1/2}^{\text{in}}\geq0$, and
$p:\Rp\times\Rp\times\Rpp\to\Rp$ is a locally Lipschitz continuous function with $p(0,0,t)=0$  for all
$t>0$. The reference signal is a constant positive function $y_{\rf}\equiv y^{*}>0$.
The system~\eqref{eq:ExampleExothermicReaction} is already given in Byrnes-Isidori form and
has global relative degree~$r=1$ with positive high-frequency gain. As in~\cite{IlchTren04} we choose for the function
$p$ the Arrhenius law $p(x_1, x_2, y)=k_0\me^{-\tfrac{k_1}{y}}x_1$ with $k_0,k_1\in\Rpp$. Since $c_1<0$, it is easy
to see that the subsystem
\begin{align*}
    \dot{x}_1(t) &= c_1 p(x_1(t), x_2(t), y(t)) +d(x_1^{\text{in}}-x_1(t)),\\
    \dot{x}_2(t) &= c_2 p(x_1(t), x_2(t), y(t)) +d(x_2^{\text{in}}-x_2(t)),
\end{align*}
satisfies the~BIBS condition~\eqref{eq:BIBO-ID}, when~$y$ is restricted to the set $\setdef{\! y\in
W^{1,\infty}(\Rp,\R)}{\forall\, t\ge 0:\ y(t)>0\!}$. We like to emphasize that the control must
guarantee that $y$ is always positive. The objective is to track the reference signal~$y_{\rf}$ by
application of the FMPC Algorithm~\ref{Algo:MPCFunnelCost} such that for a given $\phi\in\cG$ the
error $e:=y-y_{\rf}$ evolves within the prescribed performance funnel, i.e.,
$\phi(t)\Norm{e(t)} < 1$ for all $t\geq 0$. 

For the simulation we choose the funnel function $\varphi\in\cG$ given by $\phi(t) = \rbl100\me^{-2t}+1.5\rbr^{-1}$, $t\ge 0$, and allow a maximal control value of
$M = 600$, i.e., the input constraints are ${\SNorm{u}\leq600}$. As in~\cite{IlchTren04}, the initial
data is $(x_1^0, x_2^0, y^0) = (0.02,0.9,270)$, the reference
signal is $y_{\rf}\equiv y^{\star}=337.1$ and the parameters are
\[
     c_1=-1,\quad
     c_2=1,\quad
     k_0=\me^{25},\quad
     k_1=8700,\quad
     d= 1.1\quad
     q=1.25,\quad
     x_{1}^{\text{in}}=1,\quad
     x_{2}^{\text{in}}=0,\quad
     b=209.2.
\]
Due to discretisation, only step functions with constant
step length~$0.05$ were considered\footnote{
    By a step function on an interval~$[a,b]$ with constant step length~$\delta>0$, we mean a
    mapping $f:[a,b]\to\R$ which is constant on every interval $[a+k\delta,a+(k+1)\delta)\cap[a,b]$ for
    $k=0,\ldots,\lceil\frac{b-a}{\delta}\rceil-1$.
}
for the OCP~\eqref{eq:FunnelMpcOCP} of  the
FMPC Algorithm~\ref{Algo:MPCFunnelCost}. The prediction horizon and time shift are selected as $T=0.5$ and
$\delta=0.05$, resp. We further choose the parameter $\l_u=1$ for the stage cost
~$\ell_{\phi}$ given by~\eqref{eq:stageCostFunnelMPC}. The simulation was performed on the
time interval $[0,4]$ with the MATLAB routine~\texttt{ode45}. Although the considered step length is
relatively large, the FMPC Algorithm~\ref{Algo:MPCFunnelCost} achieves the control objective without further
tuning of the parameter $\l_u$. The simulation of the FMPC Algorithm~\ref{Algo:MPCFunnelCost}
applied to the model~\eqref{eq:ExampleExothermicReaction} is depicted
in Figure~\ref{Fig:Ex:ExothermicReaction}. While Figure~\ref{Fig:Ex:ExothermicReaction:Output} shows
the output of the system evolving within the funnel boundaries,
Figure~\ref{Fig:Ex:ExothermicReaction:Input} shows the corresponding input signal.

\begin{figure}[h]
    \centering
    \begin{subfigure}{.5\textwidth}
      \centering
      \includegraphics[width=\linewidth]{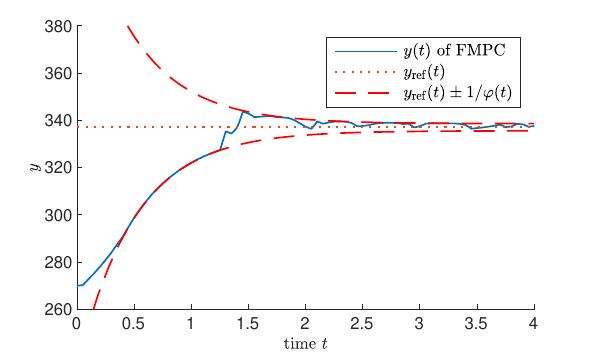}
      \caption{Funnel and system output}
      \label{Fig:Ex:ExothermicReaction:Output}
    \end{subfigure}%
    \begin{subfigure}{.5\textwidth}
      \centering
      \includegraphics[width=\linewidth]{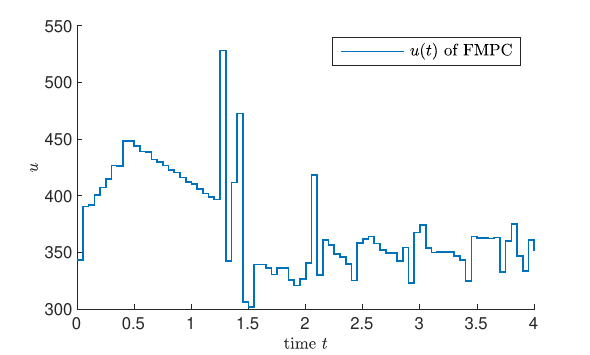}
      \caption{Input signal}
      \label{Fig:Ex:ExothermicReaction:Input}
    \end{subfigure}
    \caption{
        Simulation of system~\eqref{eq:ExampleExothermicReaction} under
        the feedback law~\eqref{eq:FMPC-fb} of the FMPC Algorithm~\ref{Algo:MPCFunnelCost}.
    }
    \label{Fig:Ex:ExothermicReaction}
\end{figure}

Figure~\ref{Fig:Ex:ExothermicReaction_MPC} shows the system output and the control signal if the
classical~MPC scheme~\ref{Algo:MPC} with cost function $\ell$ as in \eqref{eq:stageCostClassicalMPC}
and constraints~\eqref{eq:ConstraintClassicalMPC} is applied to
system~\eqref{eq:ExampleExothermicReaction} instead of the FMPC Algorithm~\ref{Algo:MPCFunnelCost},
with the same parameters, prediction horizon and discretisation.

\begin{figure}[h]
    \centering
    \begin{subfigure}{.5\textwidth}
      \centering
      \includegraphics[width=\linewidth]{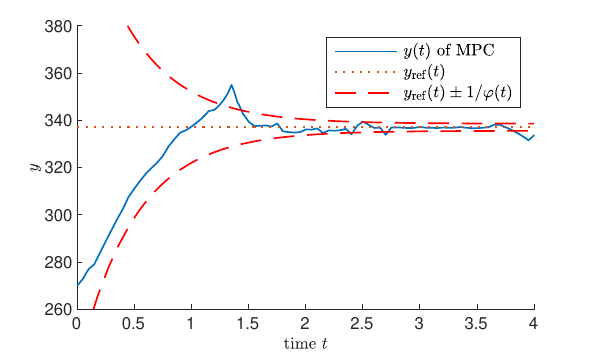}
      \caption{Funnel and system output}
      \label{Fig:Ex:ExothermicReaction_MPC:Output}
    \end{subfigure}%
    \begin{subfigure}{.5\textwidth}
      \centering
      \includegraphics[width=\linewidth]{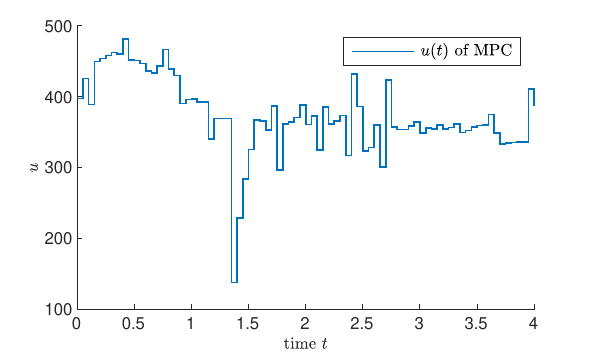}
      \caption{Input signal}
      \label{Fig:Ex:ExothermicReaction_MPC:Input}
    \end{subfigure}
    \caption{
        Simulation of system~\eqref{eq:ExampleExothermicReaction} under the classical~MPC
        scheme~\ref{Algo:MPC}.
    }
    \label{Fig:Ex:ExothermicReaction_MPC}
\end{figure}

This control does not achieve the control objective since the tracking error  exceeds the funnel
boundaries. Further adaptation of the parameter $\l_u$ is necessary in order to ensure that MPC with
the corresponding OCP~\eqref{eq:MpcOCP} is feasible with this prediction horizon and
discretisation. Such tuning of parameters in order to guarantee feasibility is not necessary for the
FMPC Algorithm~\ref{Algo:MPCFunnelCost} since the stage cost function~$\ell_{\phi}$ is automatically
increasing, if the tracking error is close to the funnel boundary.

The original funnel controller proposed in~\cite{IlchRyan02b} takes the form
\begin{equation}\label{eq:Ex:ExothermicReaction:FC}
    \begin{aligned}
        u(t)&=-\frac{1}{1-\phi(t)^2\Norm{e(t)}^2}e(t).
    \end{aligned}
\end{equation}
To compare the funnel controller~\eqref{eq:Ex:ExothermicReaction:FC} with the FMPC
Algorithm~\ref{Algo:MPCFunnelCost}, we chose the prediction horizon and  time shift as $T=1$ and
$\delta=0.1$, resp. Further, the parameter $\l_u=\tfrac{1}{10}$ for the cost functional
$\ell_{\phi}$ and a maximal control value of $M = 600$ were selected.

\begin{figure}[h]
    \centering
    \begin{subfigure}{.5\textwidth}
      \centering
      \includegraphics[width=\linewidth]{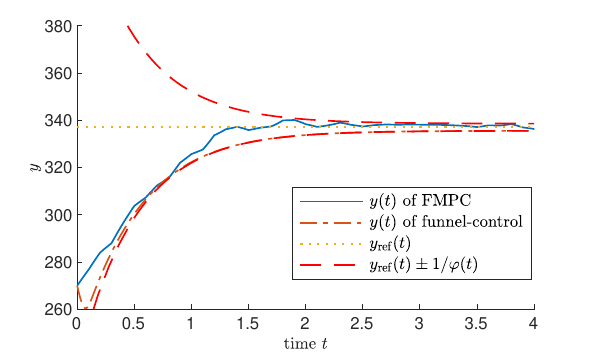}
      \caption{Funnel and tracking errors}
      \label{Fig:Ex:ExothermicReaction_FC:Output}
    \end{subfigure}%
    \begin{subfigure}{.5\textwidth}
      \centering
      \includegraphics[width=\linewidth]{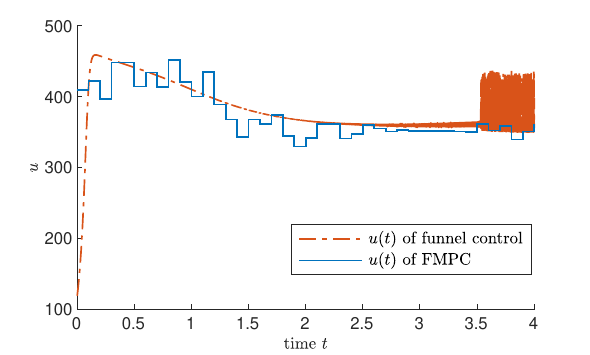}
      \caption{Input functions}
      \label{Fig:Ex:ExothermicReaction_FC:Input}
    \end{subfigure}
    \caption{
        Simulation of system~\eqref{eq:ExampleExothermicReaction} under
        controller~\eqref{eq:Ex:ExothermicReaction:FC} and FMPC Algorithm~\ref{Algo:MPCFunnelCost}.
    }
    \label{Fig:Ex:ExothermicReaction_FC}
\end{figure}

The performance of the funnel controller~\eqref{eq:Ex:ExothermicReaction:FC} and the FMPC
Algorithm~\ref{Algo:MPCFunnelCost} is depicted in Figure~\ref{Fig:Ex:ExothermicReaction_FC}. While
Figure \ref{Fig:Ex:ExothermicReaction_FC:Output} shows the tracking error of the two controllers
evolving within the funnel boundaries, Figure~\ref{Fig:Ex:ExothermicReaction_FC:Input} shows the
respective input signals. It is evident that both control techniques are feasible and achieve the
control objective.  The input signal of the funnel controller starts to oscillate at $t=2$ and
the amplitude of this oscillation increases abruptly at $t=3.5$. This behaviour is caused by a too
low sampling rate of the control signal. A relative error tolerance~(\texttt{RelTol}) of
$2.5\cdot10^{-7}$ was used. With an even higher sampling rate, this oscillation can be avoided. If a
larger error tolerance is used instead, this oscillation behaviour becomes worse. The funnel controller
becomes infeasible if the sampling rate is too low~(\texttt{RelTol} $>8\cdot10^{-6}$). The FMPC
Algorithm~\ref{Algo:MPCFunnelCost} does not show this problematic behaviour. Although FMPC uses
a relatively  wide step of $0.1$ and therefore adapts its control signal significantly less often
than the funnel controller, FMPC is feasible and the tracking error evolves within the performance funnel.



\subsection{Mass-on-car system}\label{SubSec:MassOnCar}

In this section we like to present some promising preliminary results on an extension of the FMPC
Algorithm~\ref{Algo:MPCFunnelCost} to a larger systems class.
For that we introduce the general notion of relative degree for system~\eqref{eq:Sys}.
Recall that the Lie derivative of~$h$ along~$f$ is defined by
\[
    \rbl L_f h\rbr(x)
        = \left(\sum_{i=1}^{n} \frac{\partial h_j}{\partial x_i}(x)\, f_i(x)\right)_{j=1,\ldots,n}
        = h'(x) f(x),
\]
and we may successively define~$L_f^k h = L_f (L_{f}^{k-1} h)$ with~$L_f^0 h = h$.
Furthermore, for the matrix-valued function~$g$ we have
\[
    (L_gh)(x) = \sbl (L_{g_1}h)(x), \ldots, (L_{g_m}h)(x) \sbr,
\]
where~$g_i$ denotes the~$i$-th column of~$g$ for $i=1,\ldots, m$.
Then system~\eqref{eq:Sys} is said to have \emph{(global) relative degree}~$r \in \mathbb{N}$, if
\begin{align*}
    \fa k \in \{1,\ldots,r-1\}\fa x \in \R^n:\
        (L_g L_f^{k-1} h)(x)  = 0  \quad \wedge \quad
        (L_g L_f^{r-1} h)(x)  \in \GL_m(\R),
\end{align*}
see~\cite[Sec. 5.1]{Isid95}.
The generalised \emph{high-frequency gain matrix} is defined as
\begin{equation}
    \G(x) := \rbl L_g L_f^{r-1} h\rbr(x) \in \GL_m(\R),\quad x\in\R^n.
\end{equation}
\begin{example}
    The linear system~\eqref{eq:LinearSys} of Example~\ref{Example:LinearSystem}
    has global relative degree~$r\in\N$ with~${r>1}$, if 
    \[
        \fa\;k \in \{1,\ldots,r-1\}:\
            CA^{k-1}B  = 0  \quad \wedge \quad
            CA^{r-1}B  \in \GL_m(\R).
    \]
    In other words, the relative degree is the number of times the output has to be differentiated in order
    for the input to appear explicitly on the right side of the equation.
\end{example}

For purposes of illustration that Funnel MPC shows promising results for this larger class of
systems with fixed relative degree $r\in\N$ we consider the example of a mass-spring system mounted
on a car from~\cite{SeifBlaj13} and compare FMPC with the funnel controller presented
in~\cite{BergIlch21}. This example was also examined in~\cite{BergIlch21}
and~\cite{berger2019learningbased} to compare different versions of funnel control. The mass $m_2$
moves on a ramp inclined by the angle $\vartheta\in[0,\frac{\pi}{2})$ and mounted on a car with mass
$m_1$, see Figure~\ref{Mass.on.car}.
\begin{figure}[htp]
    \begin{center}
    \includegraphics[trim=2cm 4cm 5cm 15cm,clip=true,width=6.5cm]{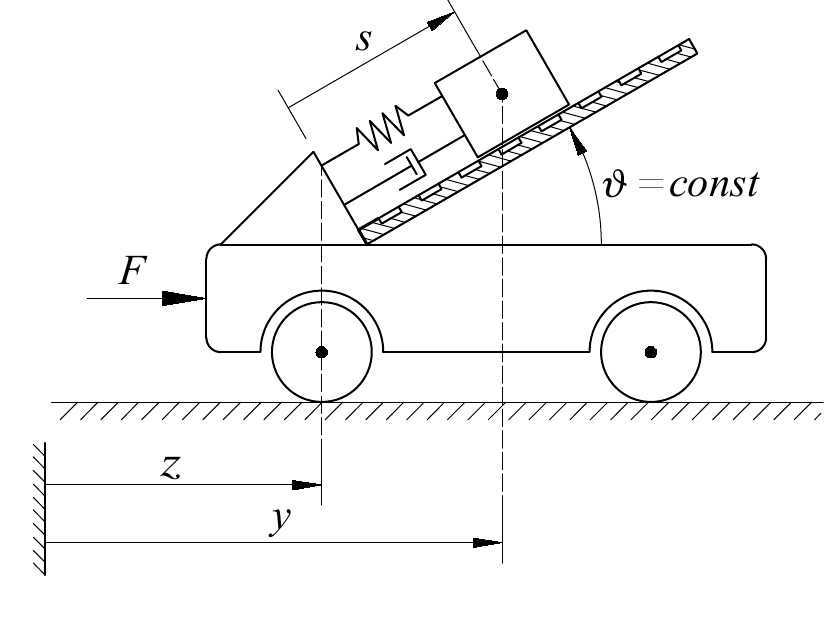}
    \end{center}
    \vspace*{-3mm}
    \caption{Mass-on-car system.}
    \label{Mass.on.car}
    \end{figure}
It is possible to control the force $u$ acting on the car. The
motion of the system is described by the equations
\begin{align}\label{eq:ExampleMassOnCarSystem}
    \begin{bmatrix}
        m_1 + m_2& m_2\cos(\vartheta)\\
        m_2 \cos(\vartheta) & m_2
    \end{bmatrix}
    \begin{pmatrix}
        \ddot{z}(t)\\
        \ddot{s}(t)
    \end{pmatrix}
    +
    \begin{pmatrix}
        0\\
        k s(t) +d\dot{s}(t)
    \end{pmatrix}
    =
    \begin{pmatrix}
        u(t)\\
        0
    \end{pmatrix},
\end{align}
where $z(t)$ is the horizontal position of the car and $s(t)$  the relative position of the mass on
the ramp at time $t$. The physical constants $k>0$ and $d>0$ are the coefficients of the spring and
damper, resp. The horizontal position of the mass on the ramp is the output $y$ of the system, i.e.,
\[
    y(t)=z(t)+s(t)\cos(\vartheta).
\]
By setting $\mu:=m_2 ( m_1 + m_2\sin^2(\vartheta))$, $\mu_1=\frac{m_1}{\mu}$, and
$\mu_2=\frac{m_2}{\mu}$, the system takes the form~\eqref{eq:LinearSys}, with
\[
    x(t):=
    \begin{pmatrix}
        z(t)\\
        \dot{z}(t)\\
        s(t)\\
        \dot{s}(t)\\
    \end{pmatrix}
    \!,\,
    A:=
    \begin{bmatrix}
        0& 1& 0& 0\\
        0& 0& \mu_2k\cos(\vartheta)& \mu_2d\cos(\vartheta)\\
        0& 0& 0& 1\\
        0& 0&-(\mu_1+\mu_2)k& -(\mu_1+\mu_2)d\\
    \end{bmatrix}
    \!,\,
    B:=
    \begin{bmatrix}
        0\\
        \mu_2\\
        0\\
        -\mu_2\cos(\vartheta)\\
    \end{bmatrix}
    \!,\,
    C:=
    \begin{bmatrix}
        1\\
        0\\
        \cos(\vartheta)\\
        0\\
    \end{bmatrix}^\top\!.
\]
It is easy to see that the system has global relative degree $r$ with
\[
    r =
        \begin{dcases}
            2, & \vartheta\in\rbl0,\tfrac{\pi}{2}\rbr\\
            3, & \vartheta=0
        \end{dcases}
\]
and the scalar high-frequency gain $\G=CA^{r-1}B$ is positive.

We choose the same parameters $m_1=4$, $m_2=1$, $k=2$, $d=1$, and initial values
$z(0)=s(0)=\dot{z}(0)=\dot{s}(0) = 0$ as in~\cite{BergIlch21}. The objective is tracking of the
reference signal $y_{\rf}:t\mapsto \cos(t)$, such that for $\phi\in\cG$ the error function $t\mapsto
e(t):=y(t)-y_{\rf}(t)$ evolves within the prescribed performance funnel, i.e., $\phi(t)\Norm{e(t)} <
1$ for all $t\geq 0$.\\

\noindent
\textbf{Case 1}: If $0<\vartheta<\frac{\pi}{2}$, then the system~\eqref{eq:ExampleMassOnCarSystem} has
relative degree $r=2$. The funnel controller presented in~\cite{BergIlch21} takes the form
\begin{equation}\label{eq:Ex:MassOnCarRelDeg2:FC}
    \begin{aligned}
        w(t)&=\phi(t)\dot{e}(t)+\alpha(\phi(t)^2e(t)^2)\phi(t)e(t),\\
        u(t)&=-\alpha(w(t)^2)w(t),
    \end{aligned}
\end{equation}
with $\alpha(s)=\frac{1}{1-s}$ for $s\in[0,1)$.
Due to discretisation, only step functions with constant step length~$0.04$ are considered for the
OCP~\eqref{eq:FunnelMpcOCP} of  the FMPC Algorithm~\ref{Algo:MPCFunnelCost}. The prediction horizon
and  time shift are selected as $T=0.6$ and $\delta=0.04$, resp. We further choose the parameter
$\l_u=\tfrac{1}{100}$ for the stage cost~$\ell_{\phi}$ and allow a maximal control value of $M
= 30$. As in~\cite{BergIlch21}, the funnel function $\phi(t) = \rbl5\me^{-2t}+0.1\rbr^{-1}$, $t\ge
0$, is chosen and the case $\vartheta=\tfrac{\pi}{4}$ is considered. All simulations are performed
on the time interval $[0,10]$ with the MATLAB routine \texttt{ode45}.

\begin{figure}[h]
    \centering
    \begin{subfigure}{.5\textwidth}
      \centering
      \includegraphics[width=\linewidth]{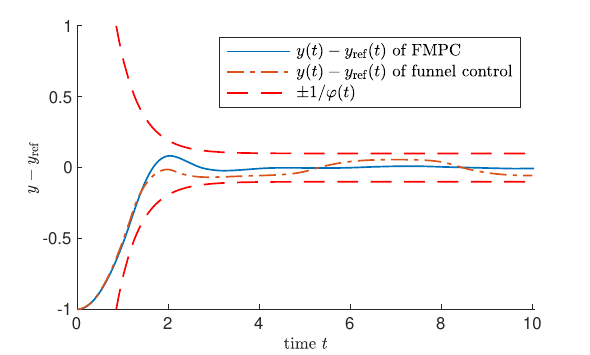}
      \caption{Funnel and tracking errors}
      \label{Fig:Ex:MaxOnCarRelDeg2:Output}
    \end{subfigure}%
    \begin{subfigure}{.5\textwidth}
      \centering
      \includegraphics[width=\linewidth]{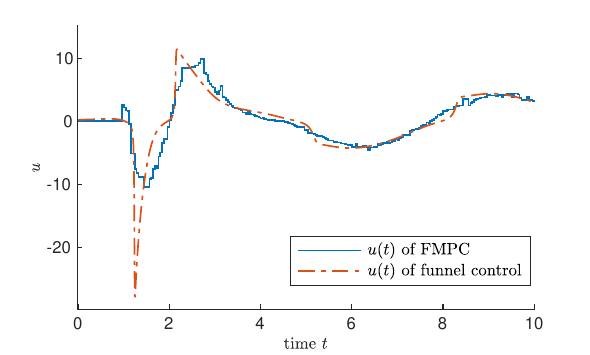}
      \caption{Input functions}
      \label{Fig:Ex:MaxOnCarRelDeg2:Input}
    \end{subfigure}
    \caption{
        Simulation of system~\eqref{eq:ExampleMassOnCarSystem} with $\vartheta=\tfrac{\pi}{4}$ under
        controller~\eqref{eq:Ex:MassOnCarRelDeg2:FC} and FMPC Algorithm~\ref{Algo:MPCFunnelCost}.
    }
    \label{Fig:Ex:MassOnCarRelDeg2}
\end{figure}

The performance of the funnel controller~\eqref{eq:Ex:MassOnCarRelDeg2:FC} and the FMPC
Algorithm~\ref{Algo:MPCFunnelCost} is depicted in Figure~\ref{Fig:Ex:MassOnCarRelDeg2}. While Figure
\ref{Fig:Ex:MaxOnCarRelDeg2:Output} shows the tracking error of the two controllers evolving within
the funnel boundaries, Figure~\ref{Fig:Ex:MaxOnCarRelDeg2:Input} shows the respective input signals.
It is evident that both control techniques are feasible and achieve the control objective.
The funnel controller is able to generate a smooth input signal,
while the~OCP~\eqref{eq:FunnelMpcOCP} of the FMPC Algorithm~\ref{Algo:MPCFunnelCost}
is optimized over step functions with constant step length~$0.04$. Nevertheless, it seems that the
FMPC Algorithm achieves a more accurate tracking of the reference signal~$y_{\rf}$ and, at the same
time, exhibits a smaller range of employed control values. Funnel control tends to change the
control values very quickly and the control signal shows spikes. The FMPC algorithm, however, avoids
this due to prediction of the future system behaviour. Similarly to~\cite{berger2019learningbased},
we observed that feasibility of the funnel controller~\eqref{eq:Ex:MassOnCarRelDeg2:FC} is not
maintained for a sampling rate $\tau =\frac{1}{300}$. Instead $\tau =\frac{1}{500}$ turns out
to be sufficient. The FMPC Algorithm~\ref{Algo:MPCFunnelCost} is feasible for both sampling rates.
Since the funnel controller needs a far higher sampling rate than~FMPC and needs to be able to adapt
its control signal very quickly, whereas FMPC uses constant steps with a relatively long
length, funnel control exhibits more demanding hardware requirements to stay feasible in
application than FMPC

When the classical~MPC Algorithm~\ref{Algo:MPC} with OCP~\eqref{eq:MpcOCP} is applied to the
system~\eqref{eq:ExampleMassOnCarSystem} with the same parameters, prediction rate and step length
instead of the FMPC Algorithm~\ref{Algo:MPCFunnelCost}, then the tracking error leaves the
performance funnel and hence the control objective is not achieved (see
Figure~\ref{Fig:Ex:MassOnCarRelDeg2_MPC}). Furthermore, the control signal exhibits quite severe
peaks.

\begin{figure}[h]
    \centering
    \begin{subfigure}{.5\textwidth}
      \centering
      \includegraphics[width=\linewidth]{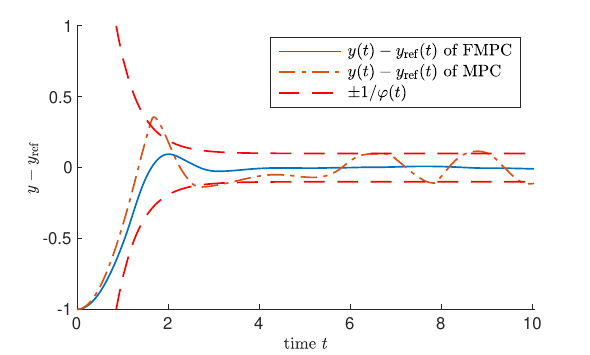}
      \caption{Funnel and tracking errors}
      \label{Fig:Ex:MaxOnCarRelDeg2_MPC:Output}
    \end{subfigure}%
    \begin{subfigure}{.5\textwidth}
      \centering
      \includegraphics[width=\linewidth]{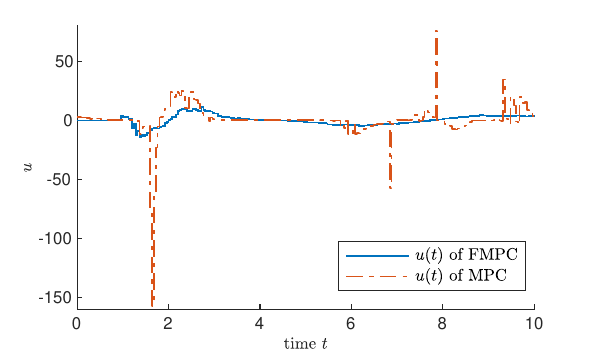}
      \caption{Input functions}
      \label{Fig:Ex:MaxOnCarRelDeg2_MPC:Input}
    \end{subfigure}
    \caption{
        Simulation of system~\eqref{eq:ExampleMassOnCarSystem} with $\vartheta=\tfrac{\pi}{4}$ under
        a classical~MPC scheme~\ref{Algo:MPC} with OCP~\eqref{eq:MpcOCP} and the FMPC
        Algorithm~\ref{Algo:MPCFunnelCost}.
    }
    \label{Fig:Ex:MassOnCarRelDeg2_MPC}
\end{figure}

A possible explanation may be that the constraint $\Norm{y(t)-y_{\rf}(t)} \leq \frac{1}{\phi(t)}$ of
the OCP~\eqref{eq:MpcOCP} does not influence the control value as long as it is
satisfied, and when the error is close to the funnel boundary, it is too late for the controller to
react. The controller attempts to compensate this by generating very large control signals. The FMPC
algorithm is able to avoid this behaviour by reacting in advance to a close funnel boundary, because
a small distance is penalised by the stage cost. Further adaptation of the parameter
$\l_u$, a smaller step length, or a longer prediction horizon are necessary in order to guarantee
feasibility of the classical~MPC scheme~\ref{Algo:MPC}.
Figure~\ref{Fig:Ex:MassOnCarRelDeg2_MPC_tuned} depicts the simulation of the classical~MPC scheme
with such adapted parameter ($\l_u=\tfrac{1}{4450}$) in comparison to the FMPC algorithm with the
same parameters as before. With these tuned parameters, the classical~MPC~\ref{Algo:MPC} scheme
 achieves the control objective and the error evolves within the funnel boundaries.

\begin{figure}[h]
    \centering
    \begin{subfigure}{.5\textwidth}
      \centering
      \includegraphics[width=\linewidth]{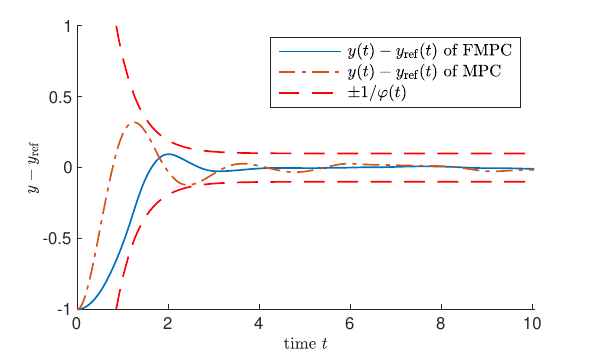}
      \caption{Funnel and tracking errors}
      \label{Fig:Ex:MaxOnCarRelDeg2_MPC_tuned:Output}
    \end{subfigure}%
    \begin{subfigure}{.5\textwidth}
      \centering
      \includegraphics[width=\linewidth]{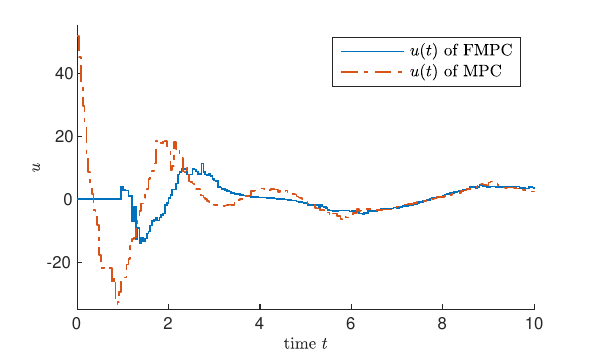}
      \caption{Input functions}
      \label{Fig:Ex:MaxOnCarRelDeg2_MPC_tuned:Input}
    \end{subfigure}
    \caption{
        Simulation of system~\eqref{eq:ExampleMassOnCarSystem} with $\vartheta=\tfrac{\pi}{4}$ under
        the FMPC Algorithm~\ref{Algo:MPCFunnelCost} and classical~MPC~\ref{Algo:MPC} 
        with OPC~\eqref{eq:MpcOCP} and $\l_u = \frac{1}{4450}$.
    }
    \label{Fig:Ex:MassOnCarRelDeg2_MPC_tuned}
\end{figure}

\noindent
\textbf{Case 2}:
If $\vartheta=0$, then the system~\eqref{eq:ExampleMassOnCarSystem} has
relative degree $r=3$. The funnel controller from~\cite{BergIlch21} takes the form
\begin{equation}\label{eq:Ex:MassOnCarRelDeg3:FC}
    \begin{aligned}
        w(t)&=\phi(t)\ddot{e}(t)+\g(\phi(t)\dot{e}(t)+\g(\phi(t)e(t))),\\
        u(t)&=-\g(w(t)),
    \end{aligned}
\end{equation}
where $\g(s) = \frac{s}{1-s^2}$ for $s\in[0,1)$.
The OCP~\eqref{eq:FunnelMpcOCP} of  the FMPC Algorithm~\ref{Algo:MPCFunnelCost} is solved over step
functions with constant step length $\tfrac{1}{15}$. The prediction horizon is $T=1$ and the time
shift is $\delta=\tfrac{1}{15}$. We further choose the parameter $\l_u=\frac{1}{100}$ for the stage
cost~$\ell_{\phi}$ and allow a maximal control value of $M = 30$. As in~\cite{BergIlch21}, we
choose the funnel function $\phi(t) = \rbl3\me^{-t}+0.1\rbr^{-1}$, $t\ge 0$.
\begin{figure}[H]
    \centering
    \begin{subfigure}{.5\textwidth}
      \centering
      \includegraphics[width=\linewidth]{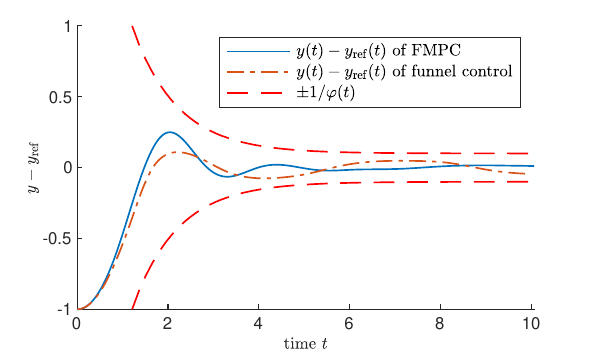}
      \caption{Funnel and tracking errors}
      \label{Fig:Ex:MaxOnCarRelDeg3:Output}
    \end{subfigure}%
    \begin{subfigure}{.5\textwidth}
      \centering
      \includegraphics[width=\linewidth]{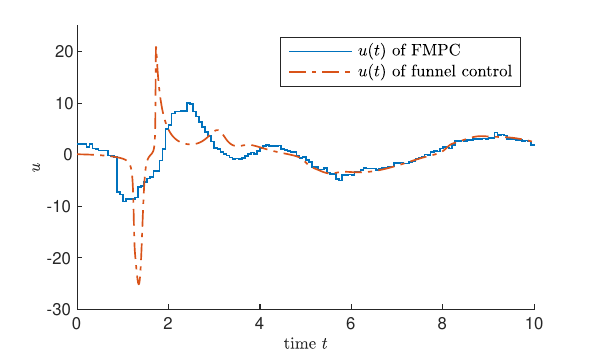}
      \caption{Input functions}
      \label{Fig:Ex:MaxOnCarRelDeg3:Input}
    \end{subfigure}
    \caption{
        Simulation of system~\eqref{eq:ExampleMassOnCarSystem} with $\vartheta=0$ under
        controller~\eqref{eq:Ex:MassOnCarRelDeg3:FC} and FMPC Algorithm~\ref{Algo:MPCFunnelCost}.
    }
    \label{Fig:Ex:MassOnCarRelDeg3}
\end{figure}
The performance of the funnel controller~\eqref{eq:Ex:MassOnCarRelDeg3:FC} and the FMPC
Algorithm~\ref{Algo:MPCFunnelCost} are depicted in Figure~\ref{Fig:Ex:MassOnCarRelDeg3}. The results
are similar to the first case with relative degree $r=2$. Note that the funnel
controllers~\eqref{eq:Ex:MassOnCarRelDeg2:FC} and~\eqref{eq:Ex:MassOnCarRelDeg3:FC} are structurally different
due to the altered relative degree $r$, whereas the FMPC Algorithm~\ref{Algo:MPCFunnelCost} is
the same in both cases. This is of particular relevance when the relative degree~$r$ is not known a priori.
The above simulations suggest that the FMPC Algorithm~\ref{Algo:MPCFunnelCost} also works for
systems with higher relative degree and exhibits a promising performance.




\section{Proof of the main result~\ref{Sec:MainResult}}\label{Sec:ProofMainResult}


\subsection{Optimal control problems with funnel-like stage costs}\label{SubSec:OCP}

Before proving initial and recursive feasibility of the Funnel MPC Algorithm~\ref{Algo:MPC}, we
show that, by using the stage cost function~$\ell_\phi$ as in $\eqref{eq:stageCostFunnelMPC}$, the
optimization problem~\eqref{eq:FunnelMpcOCP} has a solution and that this solution, if applied to the
system~\eqref{eq:Sys}, guarantees that the error~$e(t)=y(t)-y_{\rf}(t)$ evolves within the
performance funnel $\cF_\phi$. To that end we define, for $T>0$, $M>0$, $t^0\in\Rp$, and
$x^0\in\R^n$, the associated
\textit{Optimal Control Problem (OCP)}
\begin{equation}\label{eq:OptimalControlProblem}
            \mathop{\operatorname{minimize}}_{
                \substack
                {
                    u\in L^{\infty}([t^0, t^0+T],\R^{m}),\\
                    \SNorm{u}  \leq M
                }
            }\quad
                \int_{t^0}^{t^0 + T}\ell_\phi(t,x(t; t^0, x^0 ,u),u(t))\d t.
\end{equation}
If the Lebesgue integral in~\eqref{eq:OptimalControlProblem}  does not exist for some  $u\in
L^{\infty}([t^0, t^0+T],\R^{m})$ with $\SNorm{u}  \leq M$ (i.e., both the Lebesgue integrals of the
positive and negative part of $\ell_\phi(\cdot,x(\cdot; t^0, x^0 ,u),u(\cdot))$ are infinite), then
its value is treated as infinity. This may happen when $\phi(t) \Norm{h(x(t;t^0,x^0,u))-y_{\rf}(t)}
= 1$ for some $t\in[t^0,t^0+T]$. If the set of all such points does not have Lebesgue measure zero,
then the integral is treated as infinity as well.

We like to point out that there is a subtle difference between a Lebesgue integrable function (which
belongs to $L^1$) and a function for which the Lebesgue integral exists (which does not need to be
in $L^1$).
To make this difference clearer we call a measurable function $\zeta : B \to \R$ on a Borel set
$B\subseteq\R$ \textit{quasi-integrable}, if for $\zeta^+ :=
\max\{\zeta,0\}$ and $\zeta^-:=\max\{-\zeta,0\}$ at least one of the Lebesgue integrals
\[
    \int_B \zeta^+(t) \d t\qquad \text{or} \qquad \int_B \zeta^-(t) \d t
\]
is finite.

Proposition~\ref{Prop:FC} guarantees that,
if the funnel controller~\eqref{eq:FC} is applied to the system~\eqref{eq:Sys} with initial value
$x^0\in\cD^{\phi}_{t^0}$, then the tracking error  evolves in the interior of the funnel. It is not directly
clear that this also holds true if a solution of the optimization
problem~\eqref{eq:OptimalControlProblem} is applied to the system~\eqref{eq:Sys}. If the
initial error is inside the funnel, then it might still be possible that the error~$e$ touches or
even exceeds the boundary and evolves outside of the funnel boundary after some time.
In~\cite{berger2019learningbased} this issue was resolved by appending state constraints to the
optimal control problem. In the following we show that such constraints are unnecessary. In
fact, if an arbitrary control function $u\in L^\infty([t^0,t^0+T],\R^{m})$ such that
$\ell_{\phi}(\cdot,x(\cdot;t^0,x^0,u),u(\cdot))$ is quasi-integrable over $[t^0,t^0+T]$ is applied to the system,
then it is guaranteed that the error $e$ evolves within the funnel. To show this, an elementary lemma
is proved first.

\begin{lemma}\label{Lem:PosLipschitzCont}
   Let $T>0$ and $g:[0,T]\to\Rp$ be Lipschitz continuous.
    If~$\int_{0}^{T}\frac{1}{g(s)}\d s<\infty$, then $g(t)>0$ for all $t\in[0,T]$.
\end{lemma}
\begin{proof}
    First assume that there exists $\tau\in(0,T)$ such that $g(\tau)=0$. Choose $\e>0$ such that
    ${(\tau-\e,\tau+\e)\subset [0,T]}$. Since $g$ is Lipschitz continuous, we have  that
    \[
      \ex C>0 \fa s\in(\tau-\e,\tau+\e):\  g(s) =     \Abs{ g(s)-g(\tau)}
            \leq  C\Abs{ s-\tau      }.
    \]
    Therefore,
    \begin{align*}
        \infty>\int_{0}^{T}\frac{1}{g(s)}\d s
        \geq \int_{\tau-\e}^{\tau+\e}\frac{1}{g(s)}\d s
        \geq \int_{\tau-\e}^{\tau+\e}\frac{1}{C\Abs{ s-\tau}}\d s
        =    \int_{-\e}^{\e}\frac{1}{C\Abs{ s}}\d s
        =    \infty,
    \end{align*}
    a contradiction. A similar proof applies in the cases $\tau=0$ and $\tau=T$.
\end{proof}

\begin{remark}
    Lemma~\ref{Lem:PosLipschitzCont} is not true for all uniformly
    continuous functions in general. Consider the example:
    \[
        \int_{0}^1\frac{1}{\sqrt{x}}\d x
            = 2\sqrt{x}\Big|_0^1
            = 2.
    \]
\end{remark}
\begin{theorem} \label{Th:FiniteJImplFunnel}
    Consider system~\eqref{eq:Sys} with $(f,g,h)\in\cN^{m}$. Let $\phi\in\cG,\ y_{\rf}\in
    W^{1,\infty}(\Rp,\R^{m})$, $T>0$ $M>0$, $t^0\in\Rp$, and $x^0\in\R^n$ be given such that
    ${\Controls\neq \emptyset}$.
    Then the following identities hold:
    \begin{align*}
        \Controls &=
        \setdef
            {\!\! u\in \FeasibControls \!}
            {\!\!\!
                 \begin{array}{l}
                \ell_{\phi}(\cdot,x(\cdot;t^0,x^0,u),u(\cdot))\ \text{quasi-integrable on
                $[t^0,t^0+T]$},\\
                \int_{t^0}^{t^0 + T}\ell_\phi(t,x(t; t^0, x^0 ,u),u(t))\, {\rm d} t < \infty,
                \ \text{and}\ \SNorm{u} \le M \
            \end{array} \!\!\!\!
            }\\
            &=  \setdef
            {\!\! u\in \FeasibControls\!}
            {\!\!\!
                \begin{array}{l}
                    \ell_{\phi}(\cdot,x(\cdot;t^0,x^0,u),u(\cdot))\in L^1([t^0,t^0+T],\R),\\
                    \ell_{\phi}(\cdot,x(\cdot;t^0,x^0,u),u(\cdot)) \ge 0,\ \text{and}\ \SNorm{u} \le M  \\
                \end{array} \!\!\!\!
            }
    \end{align*}
\end{theorem}
\begin{proof}
Given $u\in  \Controls$, it follows from the definition of $\Controls$ that
    \[
        \phi(t)\Norm{h(x(t;t^0,x^0,u)) - y_{\rf}(t)}<1
    \]
    for all $t\in[t^0,t^0+T]$. Define $e(t) := h(x(t;t^0,x^0,u)) - y_{\rf}(t)$.
    Due to continuity of $h,\phi,y_{\rf}$, and $x(\cdot;t^0,x^0,u)$, there exists $\e\in(0,1)$ with
    $\phi(t)^2 \Norm{e(t)}^2 < 1 - \e$ for all $t \in [t^0,t^0+T]$.
    Then, $\ell_{\phi}(t,x(t;t^0,x^0,u),u(t)) \ge 0$ for all $t \in [t^0,t^0+T]$ and
    \begin{align*}
        \int_{t^0}^{t^0+T}        \Abs{ \ell_{\phi}(t,x(t;t^0,x^0,u),u(t))}\d{t}
        &=      \int_{t^0}^{t^0+T}\Abs{ \frac{1}{1-\phi(t)^2 \Norm{e(t)}^2} -1
                                        + \l_u\Norm{u(t)}^2}\d{t}\\
        &\leq   \int_{t^0}^{t^0+T}{ \frac{1}{\e} -1 + \l_u \SNorm{u}^2}\d{t} \le \left(\frac{1}{\e}
        -1 + \l_u M^2\right) T  <      \infty.
    \end{align*}
    Therefore, $\ell_{\phi}(\cdot,x(\cdot;t^0,x^0,u),u(\cdot))\in L^1([t^0,t^0+T],\R)$ and so
    $\Controls$ is contained in both of the other two sets in the statement of the theorem.

    Let $u\in \FeasibControls$ with $\SNorm{u}<M$ and quasi-integrable
    $\ell_{\phi}(\cdot,x(\cdot;t^0,x^0,u),u(\cdot))$ such that 
    $\int_{t^0}^{t^0 + T}\ell_\phi(t,x(t; t^0, x^0 ,u),u(t))\, {\rm d} t < \infty$
    be given. We now show that the error $e(t) := h(x(t;t^0,x^0,u)) - y_{\rf}(t)$ satisfies
    $\phi(t)\Norm{e(t)}<1$ for all $t\in[t^0,t^0+T]$. We
    already know $x^0\in\cD^{\phi}_{t^0}$ since ${\Controls\neq \emptyset}$, i.e.,
    $\phi(t^0)\Norm{e(t^0)}<1$.  Assume there exists $t\in(t^0,T]$ with $\phi(t)\Norm{e(t)} \geq1$.
    By continuity of $x(\cdot;t^0,x^0,u)$,$\phi$, $h$, and $y_{\rf}$ there exists
    \[
        \widehat t:=\min\setdef
            {\tau\in(t^0,t^0+T]}
            {
                \ \phi(\tau)\Norm{ e(\tau)} = 1
            }.
    \]
    Note that $\phi(t) \Norm{e(t)}< 1$ for all $t\in[t^0,\widehat t)$. Since
    $\ell_{\phi}(\cdot,x(\cdot;t^0,x^0,u),u(\cdot))$ is quasi-integrable and $\int_{t^0}^{t^0 +
    T}\ell_\phi(t,x(t; t^0, x^0 ,u),u(t))\, {\rm d} t < \infty$ it follows that  the set
    $\setdef{t\in[t^0,t^0+T]}{\phi(t) \Norm{ e(t)} = 1}$ has Lebesgue measure zero and
    \[
        \int_{t^0}^{t^0 + T} \big(\ell_\phi(t,x(t; t^0, x^0 ,u),u(t))\big)^+\, {\rm d} t < \infty.
    \]
    Therefore,
    \[
        \int_{t^0}^{t^0+T}  \left( \frac {1}{1-\phi(t)^2 \Norm{e(t)}^2} -1
                                + \l_u \Norm{u(t)}^2\right)^+ \d{t}
            = \int_{t^0}^{t^0+T}\big(\ell_\phi(t,x(t; t^0, x^0 ,u),u(t))\big)^+
        < \infty.
    \]
    Invoking $\|u\|_\infty\le M$, this yields $\int_{t^0}^{t^0+T}  \left( \frac {1}{1-\phi(t)^2
    \Norm{e(t)}^2} \right)^+ \d{t} <\infty$ and thus
    \begin{align*}
        \int_{t^0}^{\widehat t}\frac {1}{1-\phi(t)^2 \Norm{e(t)}^2}\d{t}
        =      \int_{t^0}^{\widehat t}\left( \frac {1}{1-\phi(t)^2 \Norm{e(t)}^2} \right)^+\d{t}
        \leq   \int_{t^0}^{t^0+T}     \left( \frac {1}{1-\phi(t)^2 \Norm{e(t)}^2} \right)^+\d{t}
        <      \infty.
    \end{align*}
    Since $\phi\in W^{1,\infty}(\Rp,\R)$ and $y_{\rf}\in W^{1,\infty}(\Rp,\R^{m})$, $\phi$ and
    $y_{\rf}$ are Lipschitz continuous and bounded on the interval $[t^0,\widehat t]$. Let
    $\Phi:\R^n\to\R^n$ be a diffeomorphism such that the coordinate transformation
    $\Phi(x)=(y,\eta)$ puts the system~\eqref{eq:Sys} into Byrnes-Isidori form~\eqref{eq:BIF}, then
    $\dot y$ can be written as
    \[
        \dot y(t) = p\rbl y(t),\eta(t)\rbr + \G\rbl \Phi^{-1}\rbl y(t),\eta(t)\rbr\rbr\,
        u(t),\\
    \]
    where $(y(t),\eta(t)) = \Phi(x(t;t^0,x^0,u))$ for $t\in[t^0,t^0+T]$.
    Since $\|e(t)\|\le \phi(t)^{-1}$ for all $t\in [t^0,\widehat t]$, the error $e$ is bounded and
    so $y$ is bounded, too. Hence by the~BIBS assumption~\eqref{eq:BIBO-ID}, applied to $\tilde y\in
    L^\infty([t^0,\infty),\R^m)$ defined by $\tilde y(t) = y(t)$ for $t\in[t^0,\hat t]$ and $\tilde
    y(t) = y(\hat t)$ for $t> \hat t$, yields that $\tilde \eta(\cdot) :=
    \eta(\cdot;t^0,\eta^0,\tilde y)$ is bounded and since $\tilde \eta|_{[t^0,\hat t]} =
    \eta|_{[t^0,\hat t]} $ we have that $\eta$ is bounded on $[t^0,\hat t]$. Thus, since $p$, $\G$,
    and $\Phi^{-1}$ are continuous, $\dot y$ is essentially bounded on $[t^0,\hat t]$. This implies
    the Lipschitz continuity of $y$. Products and sums of Lipschitz continuous functions on a
    compact interval are again Lipschitz continuous. Therefore, $1-\phi(\cdot)^2
    \Norm{e(\cdot)}^2=1-\phi(\cdot)^2\Norm{y(\cdot)-y_{\rf}(\cdot)}^2$ is Lipschitz continuous on
    $[t^0,\widehat t]$ and, according to Lemma~\ref{Lem:PosLipschitzCont}, strictly positive. This
    contradicts the definition of $\widehat t$. Hence $\Controls$ contains the second set in the
    statement of the theorem. Since the third set is itself contained in the second one the proof is
    complete.
\end{proof}

We are now in the position to define for $T>0$, $t^0\in\Rp$, $x^0\in\R^n$, and $\ell_{\phi}$ as
in~\eqref{eq:stageCostFunnelMPC} the \textit{cost functional}
\begin{equation}\label{eq:DefCostFunctionJ}
    \begin{aligned}
        &J^{\phi}_T(\cdot;t^0,x^0):L^\infty([t^0,t^0+T],\R^{m})\to\R\cup\{\infty\},\\
        &u\mapsto
        \begin{dcases}
            \int_{t^0}^{t^0+T}\!\!\!\!\!\!\ell_{\phi}(t,x(t;t^0,x^0,u),u(t)) \d t,  & 
                 \begin{array}{l}
                     \text{$u\in \FeasibControls$ and}\\
                    \text{$\ell_{\phi}(\cdot,x(\cdot;t^0,x^0,u),u(\cdot))$ quasi-integrable}
                \end{array}\\
            \infty,                             &\text{otherwise.}
        \end{dcases}
    \end{aligned}
\end{equation}

Although we know that for every $u\in L^\infty([t^0,t^0+T],\R^{m})$ there exists a unique maximal
solution $x(\cdot;t^0,x^0,u):[t^0,\omega)\to\R^n$ of the system~\eqref{eq:Sys}, this solution might
have finite escape time even before~$t^0+T$, i.e., $\omega<t^0+T$. In this case, and whenever the
stage costs $\ell_{\phi}(\cdot,x(\cdot;t^0,x^0,u),u(\cdot))$ are not quasi-integrable,
$J^{\phi}_T(u;t^0,x^0)=\infty$. In the following remark we state some immediate consequences of this
definition and Theorem~\ref{Th:FiniteJImplFunnel}.

\begin{remark}\label{Rem:PropertiesJ}
The  following statements hold under the assumptions of Theorem~\ref{Th:FiniteJImplFunnel}:
    \begin{enumerate}[(i)]
        \item $0\le J^{\phi}_T(u;t^0,x^0) <\infty$ \ for all $u \in \Controls$.
        \item
        $\Controls =
        \setdef
            { u\in L^\infty([t^0,t^0+T],\R^m) }
            { \SNorm{u} \le M,\ J^{\phi}_T(u;t^0,x^0)<\infty }$.
    \item The optimal control problem~\eqref{eq:OptimalControlProblem} can be reformulated as
        \[
            \mathop{\operatorname{minimize}}_{
                \substack
                {
                    u\in L^{\infty}([t^0, t^0+T],\R^{m}),\\
                    \SNorm{u}  \leq M
                }
            }
             J^{\phi}_T(u;t^0,x^0).
        \]
    \end{enumerate}
\end{remark}

\begin{remark}
    As opposed to FMPC, barrier function based MPC (see e.g.~\cite{WILLS20041415})  uses
    (relaxed) logarithmic barrier functions to penalise states close to the boundaries of the
    constraints. Although this might seem to be a subtle difference, this choice has
    remarkable implications. Lemma~\ref{Lem:PosLipschitzCont} is a consequence of the
    non-integrability of $\tfrac{1}{x}$ over the interval~$[0,1]$. As pointed out in
    Remark~\ref{Rem:PropertiesJ}, as result of this, a finite value of the cost function ensures
    that the tracking error~$e:=y-y_{\text{ref}}$ remains within the prescribed funnel
    boundaries.
    The logarithm on the other hand is integrable over the interval~$[0,1]$:
    \[
        \int_{0}^{1}\ln(x^n)\text{d} x= 
        \underbrace{\ln(x^n)x\Big\vert_{0}^{1}}_{=0} -\int_{0}^{1}x\frac{n}{x}\text{d} x=-n.
    \]
    Therefore, such a cost function alone can in general not guarantee that the state always remain within
    the desired region and therefore the usage of terminal conditions (costs and constraints)
    remains necessary.
\end{remark}

If the initial value $x^0$ is within the set $\cD^{\phi}_{t^0}$, then any control $u$ with
$J^{\phi}_T(u;t^0,x^0)<\infty$ guarantees that, if applied to the system~\eqref{eq:Sys}, the error
$e(t)=y(t)-y_{\rf}(t)$ remains strictly within the funnel. Since $J^{\phi}_T(u;t^0,x^0)$ is positive
for all control functions $u\in \Controls$, this raises the question as to whether there exists an
optimal  $u^{\star}$ which minimizes $J^{\phi}_T(\cdot;t^0,x^0)$ and is a solution to the optimal
control problem~\eqref{eq:OptimalControlProblem}. The answer is affirmative and shown in the next
theorem.



\begin{theorem}\label{Th:Funnel_cost_l2}
    Consider system~\eqref{eq:Sys} with $(f,g,h)\in\cN^{m}$. Let $\phi\in\cG,\ y_{\rf}\in
    W^{1,\infty}(\Rp,\R^{m})$, $T>0$, $M>0$, $t^0\in\Rp$, and $x^0\in\cD^{\phi}_{t^0}$ such that
    ${\Controls\neq \emptyset}$. 
    Then, there exists a function $u^{\star}\in \Controls$ such that
    \[
        J^{\phi}_T(u^{\star};t^0,x^0)=\min_{u\in \Controls}J^{\phi}_T(u;t^0,x^0) = \min_{
                \substack
                {
                    u\in L^{\infty}([t^0, t^0+T],\R^{m}),\\
                    \SNorm{u}  \leq M
                }
            }
             J^{\phi}_T(u;t^0,x^0).
    \]
\end{theorem}
\begin{proof}
    The proof essentially follows the lines of~\cite[Prop.~2.2]{sakamoto2020}.\\
    To simplify the notation, assume without loss of generality
    that~$t^0=0$ and consider only
    the interval~$[0,T]$.
It follows  from Remark~\ref{Rem:PropertiesJ}  that
$J^{\phi}_T(u;t^0,x^0)\geq 0$ for all $ u \in
    \Controls$.
    Hence the infimum $J^{\star}: = \inf_{u \in \Controls} J^{\phi}_T(u;t^0,x^0)$ exists.
    Let $(u_k)\in(\Controls)^\N$ be a minimizing sequence, meaning $J^{\phi}_T(u_k;t^0,x^0)\to
    J^{\star}$.
    By definition of $\Controls$, we have $\SNorm{ u_k }\leq M$ for all~$k\in\N$.
    Since    $L^\infty([0,T],\R^{m})\sub L^2([0,T],\R^{m})$,
    we conclude  that $(u_k)$ is a bounded sequence in the Hilbert space~$L^2$, thus there
    exists a function $u^{\star}\in L^2([0,T],\R^{m})$ and a weakly convergent subsequence
    $u_k\rightharpoonup u^{\star}$ (which we do not relabel). More precisely, $u_k|_{[0,t]}
    \rightharpoonup u^{\star}|_{[0,t]}$ weakly in $L^2([0,t],\R^m)$ for all $t\in [0,T]$ as a
    straightforward argument shows.  We define $(x_k):=(x(\cdot;t^0,x^0,u_k))\in C([0,T],\R^n)^\N$
    as the sequence of associated responses.
\\

\noindent
    \emph{Step 1}:
    We show that $(x_k)$ is uniformly bounded.   By $u_k\in \Controls$ we have
    $x_k(t)\in\cD_t^{\phi}$, i.e., ${\phi(t) \Norm{h(x_k(t))-y_{\rf}(t)} < 1}$ for all $t\in[0,T]$.
    Set $(y_k(t),\eta_k(t)) = \Phi(x_k(t))$ for $t\in[0,T]$ and $k\in\N$, where $\Phi:\R^n\to\R^n$
    is a diffeomorphism such that the coordinate transformation $\Phi(x)=(y,\eta)$ puts the
    system~\eqref{eq:Sys} into Byrnes-Isidori form~\eqref{eq:BIF}. Since~$\phi$ is positive on
    $[0,T]$, we obtain
    \[
        \fa t\in [0,T]:\quad
            \Norm{y_k(t)}
                \le \Norm{h(x_k(t))-y_{\rf}(t)} + \Norm{y_{\rf}(t)}
                \le \sup_{t\in[0,T]} \phi(t)^{-1} + \SNorm{y_{\rf}}
                 =: \tilde c_0
    \]
    and $c_0:= \tilde c_0 + \Norm{\Phi(x^0)}$ is independent of~$k$.
    Hence, by~\eqref{eq:BIBO-ID} there exists $c_1>0$ such that
    \[
        \fa y\in L^\infty(\R_{\ge 0},\R^m):\quad \|y\|_\infty\le \tilde c_0 \ \implies\ \|\eta(\cdot;0,\eta_k(0),y)\|_\infty \le c_1.
    \]
    Extending $y_k$ to $\R_{\ge 0}$ such that $\|y_k\|_\infty\le \tilde c_0$ then yields that
    $\|\eta_k(t)\|\le c_1$ for all $t\in [0,T]$. Therefore,
    \[
        x_k(t) \in \Phi^{-1}\rbl
        \setdef
        {
            \begin{pmatrix}
                z_1\\
                z_2
            \end{pmatrix}
            \in\R^{m}\times\R^{n-m}
        }
        {
            \Norm{z_1} \le \tilde c_0\ \wedge\
            \Norm{z_2} \le c_1
        }
        \rbr =: \cK
    \]
    for all $t\in[0,T]$ and all $k\in\N$, where $\cK$ is compact and independent of~$k$. Hence, $(x_k)$ is uniformly
    bounded.
\\

\noindent
    \emph{Step 2}:
    We show that $(x_k)$ is uniformly equicontinuous. Since the sequence $(u_k)$ is bounded,
    $M:=\sup_{k\in\N} \LNorm{ u_k}$ exists. Set $C_1:=\max_{x\in \cK}\Norm{ f(x)}$ and
    $C_2:=\max_{x\in \cK}\Norm{ g(x)}$, which exist by continuity of $f$ and $g$.    Now let $\e>0$
    and define $\delta:=\min\cbl1,\frac{1}{\e} (C_1+ M C_2)\cbr$. Let $k\in \N$ and
    $t_1,t_2\in[0,T]$ such that $|t_2-t_1|<\delta^2$. Then, using H\"older's inequality in the third
    estimate,
    \begin{align*}
        \Norm{x_k(t_2)-x_k(t_1)}
        &\leq \int_{t_1}^{t_2}\Norm{ f(x_k(s))}+ \Norm{g(x_k(s))} \Norm{u_k(s)} \d s\\
        &\leq C_1\Abs{t_2 - t_1} + C_2\int_{t_1}^{t_2}\Norm{u_k(s)} \d s\\
        &\leq C_1\sqrt{\Abs{t_2 - t_1}} + C_2\sqrt{\Abs{t_2 - t_1}}\LNorm{ u_k}\\
        &\leq C_1\sqrt{\Abs{t_2 - t_1}} + C_2\sqrt{\Abs{t_2 - t_1}} M\\
        &<    \delta(C_1 + M C_2) \le \e,
    \end{align*}
    which shows that $(x_k)$ is uniformly equicontinuous.
\\

\noindent
    \emph{Step 3}:
    By the Arzel\`{a}-Ascoli theorem
    there exists a function $x^{\star}\in \con([0,T],\R^n)$ and a uniformly convergent
    subsequence $x_{k}\to x^{\star}$ (which we do not relabel).
    Now we prove that $x^{\star}=x(\cdot;t^0,x^0,u^{\star})$, which means to show that
    \[
        x^{\star}(t)=x^0+\int_{0}^{t}f(x^{\star}(s))+g(x^{\star}(s)) u^{\star}(s)\d s,\quad t\in [0,T].
    \]
We have
    \[
        x_k(t)=x^0+\int_{0}^{t}f(x_k(s))+g(x_k(s))u_k(s)\d s,
        \quad  k\in\N,\ t\in [0,T],
    \]
    and since   $x_k$ in particular converges pointwise to~$x^{\star}$ and the sequence $(f(x_k))$ is uniformly bounded as $(x_k)$ is uniformly bounded and $f$ is continuous, the bounded convergence theorem gives that
\[
    \fa t\in[0,T]:\ \int_{0}^{t}f(x_k(s)) \d s \To \int_{0}^{t}f(x^{\star}(s)) \d s.
\]
Therefore, it remains to show
    \[
       \fa t\in[0,T]:\  \int_{0}^{t}g(x_k(s))u_k(s)\d s\To
       \int_{0}^{t}g(x^{\star}(s))u^{\star}(s)\d s.
    \]
    The argument~$s$ is omitted in the following. Since $g(x^{\star})$ is bounded on $[0,T]$, it is an element of $L^2([0,T],\R^{n\times m})$, thus the weak
    convergence of~$(u_k)$ implies
    \[
        \fa t\in[0,T]:\ \int_{0}^{t}g(x^{\star})u_k\d s \To \int_0^tg(x^{\star})u^{\star}\d s.
    \]
    Therefore, using H\"older's inequality  in the second estimate we obtain, for all $t\in [0,T]$,
    \begin{align*}
        \Norm{\int_{0}^{t}\!\!\!g(x_k)u_k -g(x^{\star})u^{\star}\!\d s}
        &= \Norm{\int_{0}^{t}\!\!\!g(x_k)u_k +g(x^{\star})u_k - g(x^{\star}) u_k - g(x^{\star})u^{\star}\!\d s } \\
        &\leq \int_{0}^{t}\!\!\!\Norm{ g(x_k)-g(x^{\star})} \Norm{u_k}\!\!\!\d s
        +\Norm{\int_{0}^{t}  g(x^{\star})u_k -g(x^{\star})u^{\star}\!\d s}\\
        &\leq \rbl\int_{0}^{t}\!\!\! \Norm{ g(x_k)-g(x^{\star})}^2\!\!\!\d s
            \rbr^{\!\frac{1}{2}}\!\!\rbl\int_{0}^{t}\!\!\!\Norm{u_k}^2\!\!\!\d s\rbr^{\!\frac{1}{2}}\!\!
            \!\!+\Norm{\int_{0}^{t}\!\!\!g(x^{\star})u_k -g(x^{\star})u^{\star}\!\!\d s}\\
        &\leq \sup_{m\in\N}\!\LNorm{u_m}\!\!
              \underbrace{\rbl\int_{0}^{t}\!\!\! \Norm{ g(x_k)-g(x^{\star})}^2\!\!\d
              s\!\rbr^{\!\frac{1}{2}}\!\!}_{\to 0}
        \!+\underbrace{\Norm{\int_{0}^{t}\!\!\! g(x^{\star})u_k -g(x^{\star})u^{\star}\!\d s}}_{\to 0}\! \to\! 0.
    \end{align*}

\noindent
    \emph{Step 4}:
    We show $\SNorm{ u^{\star}}\leq M$. To this end, define the sets
    \[
        A_m:=\setdef{t\in [0,T]}{\Norm{u^{\star} (t)}^2\ge M^2+\frac{1}{m}},
        \quad m\in\N.
    \]
    Let $\chi_{A_m}$ denote the indicator function of the set $A_m$, then, since $u_k\rightharpoonup
    u^{\star}$, we have that
    \[
        \langle u_k, \chi_{A_m} u^{\star}\rangle_{L^2} \to \langle u^{\star}, \chi_{A_m}
        u^{\star}\rangle_{L^2} = \|\chi_{A_m} u^{\star}\|_{L^2}^2.
    \]
    On the other hand, by the Cauchy-Schwarz inequality we have that
    \[
        \langle u_k, \chi_{A_m} u^{\star}\rangle_{L^2} \le \|\chi_{A_m} u_k\|_{L^2} \|\chi_{A_m}
        u^{\star}\|_{L^2},
    \]
    thus
    \[
        \|\chi_{A_m} u^{\star}\|_{L^2} =  \|\chi_{A_m} u^{\star}\|_{L^2}^{-1} \liminf_{k\to\infty}
        \langle u_k, \chi_{A_m} u^{\star}\rangle_{L^2} \le \liminf_{k\to\infty} \|\chi_{A_m} u_k\|_{L^2}
    \]
    and hence
    \[
        \int_{A_m} \|u^{\star}(s)\|^2 \d s \le \liminf_{k\to\infty}  \int_{A_m} \|u_k(s)\|^2 \d s.
    \]
    Since $\|u_k\|_\infty\leq M$, we then find the following for all $m\in\N$ and $k\in\N$:
    \begin{align*}
        \l\rbl A_m\rbr
        =    \int_{A_m}1\d s
        \leq m\int_{A_m}\Norm{u^{\star}(s)}^2-M^2 \d s
        \leq m\int_{A_m}\Norm{u^{\star}(s)}^2-\Norm{u_k(s)}^2  \d s,
    \end{align*}
    where $\l$ denotes the Lebesgue measure, thus
    \[
        0 \le \l\rbl A_m\rbr \le \liminf_{k\to\infty} m\int_{A_m}\Norm{u^{\star}(s)}^2-\Norm{u_k(s)}^2\  \d s \le 0.
    \]
    Due to the $\s$-continuity of $\l$ we get
    \[
        \l(\setdef
            {t\in [0,T]}
            {\Norm{u^{\star} (t)}>M})
        = \l\rbl\bigcup_{m\in\N}A_m\rbr
        = \lim_{m\to\infty}\l(A_m)=0.
    \]
    This implies $\SNorm{ u^{\star}}\leq M$. 

    \noindent
    \emph{Step 5}:
    We prove $u^{\star} \in \Controls$, which means to show~$x^{\star}(t)\in\cD_{t}^\phi$
    for all $t\in[0,T]$. 
    Assume there exists $\tau\in(0,T]$ with $\phi(\tau)\Norm{h(x^\star(\tau))-y_{\rf}(\tau)} \geq1$.
    By continuity of $x^{\star}$, $\phi$, $h$, and $y_{\rf}$ there exists
    \[
        \widehat t:=\min\setdef
            {\tau\in(0,T]}
            {
                \ \phi(\tau)\Norm{ h(x^\star(\tau))-y_{\rf}(\tau)} = 1
            }.
    \]
    $\dot{x}^{\star}$ is bounded on the interval $[0,T]$ since $g$ and $f$ are
    continuous and both $x^\star$ and $u^\star$ are bounded.
    Hence, $x^{\star}$ is Lipschitz continuous with Lipschitz constant $L^{\star}>0$.
    Define the continuously differentiable function~$\omega:[0,T]\times\cK\to\R, (t,x)\mapsto
    1-\phi(t)^2\Norm{h(x)-y_{\rf}(t)}^2$.
    Due to the compactness of $[0,T]$ and $\cK$, $\omega$ is Lipschitz continuous with Lipschitz
    constant $L_{\omega}>0$.
    We have $\omega(s,x_k(s))>0$ for all $k\in\N$ and all $s\in[0,\hat{t}]$ because
    $u_k\in\Controls$.
    Since $w(\widehat{t},x^\star(\widehat{t}))=0$, the following holds for all
    $s\in[0,\widehat{t}]$ and all $k\in\N$.
    \begin{align*}
        \omega(s,x_k(s))
        &=\Abs{\omega(s,x_k(s))}=\Abs{\omega(s,x_k(s))- w(\widehat{t},x^\star(\widehat{t}))}\\
        &\leq L_{\omega}\Norm{
            \begin{pmatrix}
                s-\widehat{t}\\
                x_k(s)- x^\star(\widehat{t})
        \end{pmatrix}}
        = L_{\omega}\Norm{
        \begin{pmatrix}
                s-\widehat{t}\\
            x_k(s)-x^\star(s)+ x^\star(s)-x^\star(\widehat{t})
        \end{pmatrix}}\\
        &\leq  L_{\omega}\Abs{s-\widehat{t}}+
        L_{\omega}\Norm{x_k(s)-x^\star(s)}+L_{\omega}L^\star\Abs{s-\widehat{t}}.
    \end{align*}
    The supremum $\sup_{k\in\N}J^{\phi}_T(u_k;t^0,x^0)<\infty$ exists because $J^{\phi}_T(u_k;t^0,x^0)\to J^{\star}$.
    Since  $\int_{0}^{\widehat{t}}\tfrac{1}{(L_{\omega}+L_{\omega}L^\star)\Abs{s-\widehat{t}}}\d
    s=\infty$, there exists $\delta>0$ with
    $\int_{0}^{\widehat{t}}\tfrac{1}{(L_{\omega}+L_{\omega}L^\star)\Abs{s-\widehat{t}}+L_{\omega}\delta}\d
    s-\widehat{t}>\sup_{k\in\N}J^{\phi}_T(u_k;t^0,x^0)$. Due to the uniform convergence of $x_k$ to $x^{\star}$, there
    exists $K\in\N$ such that $\Norm{x_k(s)-x^\star(s)}<\delta$ for all $k\geq K$ and all
    $s\in[0,\widehat{t}]$. Thus, we arrive, for $k\geq K$, at the following contradiction.
    \begin{align*}
        \sup_{k\in\N}J^{\phi}_T(u_k;t^0,x^0) 
        &\geq\int_{0}^{T}\ell_{\phi}(s,x_k(s),u_k(s))\d s \\
        &=\int_{0}^{T}
             \frac {1}{1-\phi(s)^2
             \Norm{h(x_k(s))-y_{\rf}(s)}^2}-1  + \l_u
             \Norm{u_k(s)}^2  \d s\\
        &\geq \int_{0}^{\widehat{t}}
            \frac {1}{\omega(s,x_k(s))}-1\d s\\
        &\geq \int_{0}^{\widehat{t}}
            \frac {1}{
        (L_{\omega}+L_{\omega}L^\star)\Abs{s-\widehat{t}}+
        L_{\omega}\Norm{x_k(s)-x^\star(s)}}\d s -\widehat{t}\\
        &> \int_{0}^{\widehat{t}}
            \frac {1}{
        (L_{\omega}+L_{\omega}L^\star)\Abs{s-\widehat{t}}+
        L_{\omega}\delta}\d s -\widehat{t}>\sup_{k\in\N}J^{\phi}_T(u_k;t^0,x^0).
    \end{align*}
    Hence, $u^{\star} \in \Controls$.

    \noindent
    \emph{Step 6}:
    We show $J^{\phi}_T(u^{\star};t^0,x^0) = J^{\star}$. 
    Let $\tilde\ell_{\phi}:\cD^{\varphi}\to \R, (t,x)\mapsto \tfrac {1}{1-\phi(t)^2
    \Norm{h(x)-y_{\rf}(t)}^2}-1$. 
    For all $k\in\N$, we have $\SNorm{\tilde\ell_{\phi}(\cdot,x_k(\cdot))}<\infty$
    and $\SNorm{\tilde\ell_{\phi}(\cdot,x^{\star}(\cdot))}<\infty$ because ${u_k,u^\star\in\Controls}$.
    According to Step~5, there exists $\e>0$ such that $\Norm{x^{\star}(t)}\leq
    \tfrac{1}{\phi(t)}-\e$ for all $t\in[0,T]$.
    Due to the uniform convergence of $x_k$ to $x^\star$, there exists $N\in\N$ such that 
    $\SNorm{x_k-x^{\star}}<\tfrac{\e}{2}$ for $k\geq N$. Thus,
    \[
        \fa k\geq N \fa t\in[0,T]:\quad
        \Norm{x_k(t)}\leq
        \Norm{x_k(t)-x^{\star}(t)}+\Norm{x^{\star}(t)}<\frac{1}{\phi(t)}-\frac{\e}{2}.
    \]
    Hence, the sequence $\rbl\tilde\ell_{\phi}(\cdot,x_k(\cdot))^{\tfrac12}\rbr$ is uniformly bounded.
    Due to the continuity of $\tilde\ell$, the bounded convergence theorem gives that
    $\tilde\ell_{\phi}(\cdot,x_k(\cdot))^{\tfrac12}\to\tilde\ell_{\phi}(\cdot,x^\star(\cdot))^{\tfrac12}$
    strongly and, thus, also weakly in $L^2([0,T],\R)$.
    Since $J^{\phi}_T(u_k;t^0,x^0)\to J^{\star} = \inf_{u \in \Controls} J^{\phi}_T(u;t^0,x^0)$
    and since the $L^2$-norm is weakly lower semi-continuous, the following holds.
    \begin{align*}
        J^{\phi}_T(u^{\star};t^0,x^0)&=\int_{0}^{T}\ell_{\phi}(s,x^{\star}(s),u^{\star}(s))\d s
                 = \LNorm{ \tilde\ell_{\phi}(\cdot,x^{\star}(\cdot))^\frac12}^2
                 +\l_u\LNorm{u^{\star}}^2\\
                &\leq \liminf_{k\to\infty} \LNorm{ \tilde\ell_{\phi}(\cdot,x_k(\cdot))^\frac12}^2+\liminf_{k\to\infty}
                     \l_u\LNorm{u_k}^2
                 \leq \liminf_{k\to\infty} J^{\phi}_T(u_k;t^0,x^0) = J^{\star}.
    \end{align*}
    Therefore $J^{\phi}_T(u^{\star};t^0,x^0) = \min_{u\in \Controls}J^{\phi}_T(u;t^0,x^0)$.\\

\noindent
    \emph{Step 7}: We show that $J^{\phi}_T(u^{\star};t^0,x^0) = \min_{\substack
                {
                    u\in L^{\infty}([t^0, t^0+T],\R^{m}),\\
                    \SNorm{u}  \leq M
                }
            }
             J^{\phi}_T(u;t^0,x^0)$. Since $\Controls \neq \emptyset$ by assumption this follows from Remark~\ref{Rem:PropertiesJ}\,(ii) and completes the proof.
\end{proof}



\subsection{Initial and recursive feasibility}\label{SubSec:Feasibility}

In the following we seek to show initial and recursive feasibility of the FMPC
Algorithm~\ref{Algo:MPCFunnelCost}. For this we need to show that the essential assumption
$\Controls \neq \emptyset$ of Theorem~\ref{Th:Funnel_cost_l2} is initially (i.e., at $t=t^0$) and
recursively (i.e., at $t=t^0+\delta n$ after $n$ steps of Algorithm~\ref{Algo:MPCFunnelCost})
satisfied. The main difficulty is to prove the existence of a number $M>0$ for which the latter is
satisfied for all initial values within a prescribed bounded set. This is the purpose of the
following results.

First observe that applying the funnel controller~\eqref{eq:FC} from Proposition~\ref{Prop:FC} to
the system~\eqref{eq:Sys} ensures that the error evolves strictly within the funnel for any initial
condition $x^0\in\cD^{\phi}_{t^0}$. As stated in Remark~\ref{Rem:FunnelBound} the funnel
controller is bounded. This bound however depends on the initial value $x^0$. This means that for
every~$x^0 \in \cD^{\phi}_{t^0}$ there exists $M>0$ such that $\Controls$ is non-empty. This raises
the question whether it is possible to find a bound~$M$ independent of the initial value $x^0$. The
following example shows that this is not the case in general.

\begin{example}
    Consider the two-dimensional linear system
    \begin{alignat*}{2}
        \dot{y}(t)      & = \eta(t)  +  u(t),   &\qquad    y(0)&=0, \\
        \dot{\eta}(t)   & = 0,                  &\qquad \eta(0)&=\eta^0,
    \end{alignat*}
    in Byrnes-Isidori form with constant reference signal $y_{\rf}\equiv 0$ and the constant
    funnel $\phi \equiv 1$. Let $M>0$ and $T>0$ be arbitrary. Although the system satisfies the~BIBS condition~\eqref{eq:BIBO-ID} and the initial error $e(0) = y^0 - y_{\rf}(0) = 0$ lies within the
    funnel for every $\eta^0\in\R$, there exists $\eta^0\in\R$ such that the error $e$ exceeds the
    funnel boundaries at time~$T$ for every $u\in L^\infty([0,T], \R)$ with $\|u\|_\infty\le M$. To
    see this, choose $\eta^0 := M+\frac{2}{T}$, then
    \begin{align*}
        e\rbl T\rbr
           =      y\rbl T\rbr
           =      \int_{0}^{T}\eta(s)  +  u(s)\d s
           =      T\eta^0  + \int_{0}^{T}u(s)\d s
           \geq   T\eta^0  -  T M
           =      2
           >      1
           =      \frac{1}{\phi\rbl T\rbr}.
    \end{align*}
\end{example}

The example shows that, in general, there exists no $M>0$ such that $\Controls$ is
non-empty for all $x^0 \in \cD^{\phi}_{t^0}$. However, for a bounded set $B\subset\cD^{\phi}_{t^0}$ of
initial values, it is possible to find a uniform bound~$M>0$. Moreover, $M$ can be chosen
independently  of $T>0$. To show this, we denote by $\cY^{\phi,y^0}_{y_{\rf}}(I)$ the set of all
functions in $\con(I,\R^m)$ starting at $y^0\in\R^m$ and evolving within the funnel on an interval $I\subseteq\R_{\ge 0}$ of the form $I=[a,b)$ with $b\in(a,\infty]$ or $I=[a,b]$ with $b\in (a,\infty)$:
\[
    \cY^{\phi,y^0}_{y_{\rf}}(I) := \setdef
                                {y\in \con(I,\R^m)}
                                {y(\inf I) = y^0,\ \fa t\in I:\ \phi(t)\Norm{y(t)-y_{\rf}(t)} < 1}.
\]


\begin{lemma}\label{Lemma:ExistanceCompactSet}
    Consider system~\eqref{eq:Sys} with $(f,g,h)\in\cN^{m}$. Let
    $\phi\in\cG$, $y_{\rf}\in W^{1,\infty}(\Rp,\R^{m})$,  $t^0\in\Rp$. Then for all bounded sets
    $B\subset\R^{n}$ there exists
    a compact set $K\subset\R^n$ such that
    \begin{align}\label{eq:OutPutIsInCompactSet}
        \fa T>0
        \fa (y^0,\eta^0)\in B
        \fa y\in\cY^{\phi,y^0}_{y_{\rf}}([t^0,t^0+T])
        \fa t\in[t^0,t^0+T]:\quad
            (y(t),\eta(t;t^0,\eta^0,y))\in K.
    \end{align}
\end{lemma}
\begin{proof}
    Define
    \[
        N_t:= \setdef
                {\eta(t;t^0,\eta^0,y)}
                {(y^0,\eta^0)\in B,\ y\in \cY^{\phi,y^0}_{y_{\rf}}([t^0,\infty))}
            ,\quad  t \ge t^0.
    \]
    By definition of $\cG$,   $\phi$ is strictly positive and $\inf_{t\geq 0} \phi(t) > 0$.
    Therefore, $1/\phi$ is bounded. Since clearly $y_{\rf}\in  W^{1,\infty}(\Rp,\R^{m})$ is bounded,
    every function $y\in\cY^{\phi,y^0}_{y_{\rf}}([t^0,\infty))$ is bounded by
    \[
        \SNorm{y}
        \leq \SNorm{y-y_{\rf}}+\SNorm{y_{\rf}}
        \leq \SNorm{\tfrac{1}{\phi}}+\SNorm{y_{\rf}}.
    \]
    Since $B$ is bounded it follows from the~BIBS
    condition~\eqref{eq:BIBO-ID} that  the set $N:=\bigcup_{t\geq t^0} N_t$ is also bounded. Furthermore, the set
    \[
        O:=\bigcup_{t\geq t^0}\setdef
                                        {y\in\R^m}
                                        {\phi(t)\Norm{y-y_{\rf}(t)} < 1}
    \]
    is bounded, too. Then the set $K:=\overline{O\times N}$ is compact and by definition of~$N$ and~$O$ we find that~\eqref{eq:OutPutIsInCompactSet} holds.
\end{proof}

The following result provides a number $M>0$ with  $\Controls \neq \emptyset$  for all initial values $x^0$ from any compact set which satisfies a condition similar to~\eqref{eq:OutPutIsInCompactSet}.

\begin{prop}\label{Prop:BoundM}
    Consider system~\eqref{eq:Sys} with $(f,g,h)\in\cN^{m}$. Let $\phi\in\cG$, $\ y_{\rf}\in
    W^{1,\infty}(\Rp,\R^{m})$, $T>0$, $t^0\in\Rp$, $x^0\in\cD_{t^0}^{\phi}$, and $\Phi:\R^n\to\R^n$
    be a diffeomorphism such that the coordinate transformation $\Phi(x)=(y,\eta)$ puts the
    system~\eqref{eq:Sys} into Byrnes-Isidori form~\eqref{eq:BIF}. Let $(y^0,\eta^0) = \Phi(x^0)$
    and $K\subset\R^{n}$ be a  compact set with
    \begin{equation}\label{eq:Prop:DefPropertyK}
        \fa y\in\cY^{\phi,y^0}_{y_{\rf}}([t^0,t^0+T])
        \fa t\in[t^0,t^0+T]:\quad
            (y(t),\eta(t;t^0,\eta^0,y))\in K.
    \end{equation}
     If
    \begin{equation}\label{eq:Prop:DefBoundM}
        M\geq
            G_{\rm max}
            \rbl
                P_{\rm max}+ \SNorm{\dd{t}\tfrac{1}{\phi}} + \SNorm{\dot y_{\rf}}
            \rbr,
    \end{equation}
    where, with $p(\cdot,\cdot)$ and $\Gamma(\cdot)$ as in~\eqref{eq:Gamma} and~\eqref{eq:BIF},
    \begin{align*}
        P_{\rm max} :=\max_{(y,\eta)\in K} \|p(y,\eta)\|,\quad    G_{\rm max} :=\max_{(y,\eta)\in K}
        \|\G(\Phi^{-1}(y,\eta))^{-1}\|,
    \end{align*}
    then $\Controls\neq \emptyset$.
\end{prop}
\begin{proof}
    \emph{Step 1}:
    We first show the existence of $M>0$ satisfying~\eqref{eq:Prop:DefBoundM}. Note that $p$ and
    $\G$ from~\eqref{eq:Gamma} and~\eqref{eq:BIF} are continuous and $\G$ is pointwise invertible.
    Therefore, $P_{\rm max}$ and $G_{\rm max}$ are well-defined. Furthermore, the essential supremum
    $\SNorm{\dot y_{\rf}}$ is finite, because $y_{\rf}\in W^{1,\infty}(\Rp,\R^m)$. Since $\phi$ is
    an element of $W^{1,\infty}(\Rp,\R)$, always positive and $\inf_{t\geq0} \phi(t)>0$ the
    reciprocal~$\psi := 1/\phi$  is an element of $W^{1,\infty}(\Rp,\R)$, too, and in particular
    $\dot \psi$ is bounded. Thus, $M$ can be chosen as in~\eqref{eq:Prop:DefBoundM}.
\\

\noindent
\emph{Step 2}:
We construct a control function~$u$ and show   that $u\in \Controls$.
To this end, define $e(t):=y(t)-y_{\rf}(t)$ and observe that, since
$x^0\in\cD_{t^0}^{\phi}$, we have $\phi(t^0)\Norm{e(t^0)}<1$.
The application of the output feedback
\[
    u(t) = \G\big(\Phi^{-1}(y(t),\eta(t))\big)^{-1}
                \rbl
                    - p(y(t),\eta(t))
                    + \phi(t^0)e(t^0)\dot \psi(t) +\dot{y}_{\rf}(t)
                \rbr
\]
to the system~\eqref{eq:BIF} leads to a closed-loop system. If this initial value problem is
considered on the interval~$[t^0,t^0+T]$, then there exists a unique maximal solution
$(y,\eta):[t^0,\omega)\to \R^n$  with $\omega\in (t^0,t^0+T]$ and if $(y,\eta)$ is bounded, then $\omega =  t^0+T$, cf.~\cite[\S~10,~Thm.~XX]{Walt98}. Then we find for all $t\in[t^0,\omega)$ that
\begin{align*}
    \Norm{e(t)}
        &= \Norm{
            \int_{t^0}^{t}\dot{y}(s) - \dot{y}_{\rf}(s)\d s
            + e(t^0)}\\
        &= \Norm{
            \int_{t^0}^{t}p(y(s),\eta(s)) + \G\big(\Phi^{-1}(y(s),\eta(s))\big) u(s) - \dot{y}_{\rf}(s)\d s
            + e(t^0)}\\
        &= \Norm{
            \int_{t^0}^{t}
                \phi(t^0)e(t^0) \dot \psi(s)
            \d s
            + e(t^0)}
        = \Norm{
            \phi(t^0)e(t^0)\rbl\psi(t)
            -\psi(t^0)\rbr
           + e(t^0)}\\
        &= \underbrace{\phi(t^0)\Norm{e(t^0)}}_{<1}
           \psi(t)
        < \psi(t).
\end{align*}
This means, the tracking error~$e$ remains within the
funnel, i.e., $(y(t),\eta(t))\in \cD_{t}^\phi$ for all $t\in[t^0,\omega)$. Thus, $y$ is uniformly bounded by
\[
    \SNorm{y}
    \leq \SNorm{y-y_{\rf}}+\SNorm{y_{\rf}}
    \leq \SNorm{\psi}+\SNorm{y_{\rf}}.
\]
Since the error remains within the funnel, the output $y$, defined on $[t^0,\omega)$,
can be extended to an element $\tilde y \in \cY^{\phi,y^0}_{y_{\rf}}([t^0,t^0+T])$ and so by assumption~\eqref{eq:Prop:DefPropertyK} we have
\[
    \fa t\in[t^0,\omega):\quad (y(t),\eta(t;t^0,\eta^0,y))\in K.
\]
Therefore, $(y,\eta)$ is bounded and hence $\omega =  t^0+T$ and, with the same arguments, $(y,\eta)$ has a continuous extension to $[t^0,t^0+T]$. Furthermore, by definition of~$u$ it is clear that $\|u\|_\infty\le M$ and hence $u\in \Controls$, which completes the proof.
\end{proof}

The result of Proposition~\ref{Prop:BoundM} essentially guarantees initial feasibility of
Algorithm~\ref{Algo:MPCFunnelCost} for all initial values from a given bounded set, which we will
summarize in the following theorem. To further obtain recursive feasibility we need to ensure that,
after the application of a control $u$ from $\Controls$ over an interval $[t^0,t]$, the set of
controls corresponding to the new state value, namely $\PhiTControls{t}{x(t;t^0,x^0,u)}$, is non-empty as well.

\begin{theorem}\label{Theorem:ExistanceBoundM}
    Consider system~\eqref{eq:Sys} with $(f,g,h)\in\cN^{m}$. Let $\phi\in\cG$, $\ y_{\rf}\in
    W^{1,\infty}(\Rp,\R^{m})$, $t^0\in\Rp$, and $B\sub\cD^{\phi}_{t^0}$ be a bounded set.
    Then, there exists $M>0$ such that
    \begin{align}\label{eq:ExistanceBoundMInfinite}
         \fa x^0\in B
         \fa T>0: \quad
            \Controls\neq \emptyset
    \end{align}
    and, furthermore,
    \begin{align}\label{eq:ExistanceBoundMRecursive}
        \fa x^0\in B
        \fa T_1,T_2>0
        \fa u\in \PhiControls{T_1}{t^0}{x^0}
        \fa t\in[t^0,t^0+T_1]:\quad
        \PhiControls{T_2}{t}{x(t;t^0,x^0,u)}\neq \emptyset.
    \end{align}
\end{theorem}
\begin{proof}
    Let $\Phi:\R^n\to\R^n$ be a diffeomorphism such that the coordinate transformation $\Phi(x)=(y,\eta)$ puts
    the system~\eqref{eq:Sys} into Byrnes-Isidori form~\eqref{eq:BIF}. Fix $x^0\in B$ and set $(y^0,\eta^0):=\Phi(x^0)$. According to Lemma~\ref{Lemma:ExistanceCompactSet} there exists a
    compact set $K$ such that~\eqref{eq:OutPutIsInCompactSet} holds. In particular, $K$ satisfies~\eqref{eq:Prop:DefPropertyK} for every $T>0$. Therefore, Proposition~\ref{Prop:BoundM} yields that there exists $M>0$, independent of $x^0$, such that $\Controls\neq \emptyset$ for all $T>0$, which shows~\eqref{eq:ExistanceBoundMInfinite}.

    If, for any $T_1>0$ an arbitrary but fixed control function $u\in \PhiControls{T_1}{t^0}{x^0}$ is
    applied to the system~\eqref{eq:Sys}, then  the output $y$ of the system (i.e., $y(\cdot) := h(x(\cdot;t^0,x^0,u))$) evolves within the funnel and
    is therefore an element of  $\cY^{\phi,y^0}_{y_{\rf}}([t^0,t^0+T_1])$. By~\eqref{eq:OutPutIsInCompactSet},
    this implies $\Phi(x(t;t^0,x^0,u))\in K$ for all $t\in[t^0,t^0+T_1]$. If, for any
    $\widehat t\in[t^0,t^0+T_1]$, the system is considered on the interval $[\widehat t,\widehat
    t+T_2]$ with $T_2>0$ and the current state $x(\widehat t;t^0,x^0,u)$ of the system as initial
    value, then  the prerequisites for Proposition~\ref{Prop:BoundM} are still met on the interval
    $[\widehat t,\widehat t+T_2]$, i.e., $K$ satisfies~\eqref{eq:Prop:DefPropertyK} in the sense
    \[
        \fa \tilde y \in \cY^{\phi,\widehat y}_{y_{\rf}}([\widehat t,\widehat t+T_2]) \fa t\in [\widehat t,\widehat t+T_2]:\quad (\tilde y(t), \eta(t;\widehat t,\widehat \eta, \tilde y))\in K,
    \]
    where $(\widehat y,\widehat \eta) := \Phi(x(\widehat t;t^0,x^0,u))\in K$. To see this, observe that for any
    $\tilde y \in \cY^{\phi,\widehat y}_{y_{\rf}}([\widehat t,\widehat t+T_2])$ there exists $\bar{y} \in
    \cY^{\phi,y^0}_{y_{\rf}}([t^0,\widehat t+T_2])$ with $\bar{y}|_{[\widehat t,\widehat t+T_2]} = \tilde y$ and
    $\bar{y}|_{[t^0,\widehat t]} = y$ ($\bar{y}$ is continuous since $y(\widehat t) = \widehat y$) and we have
    $\eta(t;t^0,\eta^0,\bar{y}) = \eta(t;\widehat t,\widehat \eta, \tilde y)$ for $t\in [\widehat t,\widehat t+T_2]$,
    thus the assertion follows from~\eqref{eq:Prop:DefPropertyK}. Therefore,
    Proposition~\ref{Prop:BoundM} can again be applied and yields $\PhiControls{T_2}{\widehat
    t}{x(\widehat t;t^0,x^0,u)}\neq \emptyset$, which completes the proof.
\end{proof}

\begin{example}
    We revisit Example~\ref{Example:LinearSystem} and calculate a number~$M>0$ satisfying~\eqref{eq:ExistanceBoundMInfinite} and~\eqref{eq:ExistanceBoundMRecursive} to illustrate Theorem~\ref{Theorem:ExistanceBoundM} for the linear case. Consider the system~\eqref{eq:LinearSysByrnesIsidori}
    in Byrnes-Isidori form  with $(A_1, A_2, A_3, A_4)\in
           \R^{m \times m} \times \R^{m \times (n-m) } \times
           \R^{(n-m) \times m} \times \R^{(n-m)\times (n-m)}$ and $t^0\in\Rp$.
    Let $\phi\in\cG$, $y_{\rf}\in W^{1,\infty}(\Rp,\R^{m})$ and define $\psi := 1/\phi$. Further, assume that $A_4$ is Hurwitz,  i.e., all of its eigenvalues have a negative real part.
    Then  there exist $\alpha > 0$, $\beta \geq 1$ such that
    \[
        \fa x\in\R^{n-m}\fa \text{$t \geq t^0$}:\quad
            \Norm{\me^{A_4 (t-t^0)} x}\leq\beta \me^{-\alpha (t-t^0)}\Norm{x}.
    \]
    Let $N\subset\R^{n-m}$ be an arbitrary, but fixed bounded set. We show that for
    \[
        B:=\setdef{y\in\R^m}{\phi(t^0)\Norm{y-y_{\rf}(t^0)}<1}\times N \subseteq \cD^{\phi}_{t^0}
    \]
    and
    \begin{align*}
        M:= \Norm{\G^{-1}}
            \rbl
                \rbl
                    \Norm{A_1} + \frac{\beta}{\alpha} \Norm{A_2}\Norm{A_3}
                \rbr
                \rbl
                    \SNorm{\psi} + \SNorm{y_{\rf}}
                \rbr
                +\beta \Norm{A_2} \sup_{\eta^0\in N}\Norm{\eta_0}
                +\SNorm{\dot \psi}\!\! + \SNorm{\dot{y}_{\rf}}
            \rbr
    \end{align*}
    the conditions~\eqref{eq:ExistanceBoundMInfinite} and~\eqref{eq:ExistanceBoundMRecursive} are
    satisfied. One can see that $M$ is chosen according to inequality~\eqref{eq:Prop:DefBoundM} with
    a more accurate estimate for $P_{\max}$.

    To this end, let $T>0$ and $(y^0,\eta^0)\in B$ be arbitrary and denote
    $e(t):=y(t)-y_{\rf}(t)$. Since
    the initial value is inside the funnel, we have $\phi(t^0)\Norm{e(t^0)}<1$.
    If the output feedback
    \[
        u(t):=\G^{-1}
            \rbl
                -A_1 y(t) - A_2 \eta(t)
                + \phi(t^0)e(t^0)\dot \psi(t) +\dot{y}_{\rf}(t)
            \rbr
    \]
    is applied to the system~\eqref{eq:LinearSysByrnesIsidori}, then clearly a unique global
    solution $(y,\eta):[t^0,\infty)\to\R^n$ exists and, as in the proof of
    Proposition~\ref{Prop:BoundM}, we may calculate that $ \Norm{e(t)} < \psi(t)$ for all  $t\geq
    t^0$. As a consequence $(y(t),\eta(t))\in \cD_{t}^\phi$ for all $t\ge t^0$, and $y$ is uniformly
    bounded by
    \[
        \SNorm{y}
        \leq \SNorm{y-y_{\rf}}+\SNorm{y_{\rf}}
        \leq \SNorm{\psi}+\SNorm{y_{\rf}}.
    \]
    Therefore, for $t\geq t^0$ we have
    \begin{align*}
        \Norm{\eta(t;t^0,\eta^0,y)}
            &=      \Norm {\me^{A_4(t-t^0)}\eta^0 + \int_{t^0}^{t}\me^{A_4(t-s)}A_3y(s)\d s}\\
            &\leq   \Norm {\me^{A_4(t-t^0)}}\Norm{\eta^0}
                    + \int_{t^0}^{t}\Norm{\me^{A_4(t-s)}}\Norm{A_3}\Norm{y(s)}\d s\\
            &\leq   \beta \me^{-\alpha(t-t^0)}\sup_{\eta^0\in N}\Norm{\eta^0}
                    + \int_{t^0}^{t}\beta \me^{-\alpha(t-t^0)}
                    \Norm{A_3}\rbl\SNorm{\psi}+\SNorm{y_{\rf}}\rbr\d s\\
            &\leq   \beta \sup_{\eta^0\in N}\Norm{\eta^0}
                    +\Norm{A_3}\frac{\beta}{\alpha}\rbl\SNorm{\psi}+\SNorm{y_{\rf}}\rbr.
    \end{align*}
    As a consequence we see that $\|u\|_\infty\le M$ and thus
    \[
        \fa T>0:\ u\in \PhiTControls{t^0}{(y^0,\eta^0)}\neq\emptyset
    \]
    and~\eqref{eq:ExistanceBoundMInfinite} is satisfied.
    If any $u\in \PhiTControls{t^0}{(y^0,\eta^0)}$ is applied to the
    system~\eqref{eq:LinearSysByrnesIsidori} and the system is then considered for any
    $\widehat{t}\in[t^0,t^0+T]$ and $\widehat{T}>0$ on the interval $[\widehat{t},\widehat{t}+\widehat{T}]$, then
    $\phi(\widehat{t})\Norm{e(\widehat{t})}<1$ and it can be similarly shown that the feedback control
    \[
        \widehat{u}(t):=\G^{-1}
            \rbl
                -A_1 y(t) - A_2 \eta(t)
                + \phi(\widehat{t})e(\widehat{t})\dot \psi(t) +\dot{y}_{\rf}(t)
            \rbr
    \]
    leads to an element of $\PhiControls{\widehat{T}}{\widehat{t}}{x(\widehat{t};t^0,x^0,u)}$, by
    which $M$ satisfies~\eqref{eq:ExistanceBoundMRecursive}. Here we like to emphasize that the
    estimate for $\eta$ needs to be carried out in terms of $t^0$, i.e., for $\widehat y(t):=
    h(x(t;\widehat t, x(\widehat t;t^0,x^0,u),\widehat u))$ denoting the output on
    $[\widehat{t},\widehat{t}+\widehat{T}]$ and $\widehat \eta = \eta(\widehat t;t^0,x^0,u)$ we have
    that 
    \begin{align*} 
        \eta(t;\widehat t,\widehat \eta,\widehat y) &= \me^{A_4(t-\widehat
    t)}\widehat \eta + \int_{\widehat t}^{t}\me^{A_4(t-s)}A_3\widehat y(s)\d s \\ &=
    \me^{A_4(t-t^0)}\eta^0 + \int_{t^0}^{\widehat t}\me^{A_4(t-s)}A_3y(s)\d s +  \int_{\widehat
    t}^{t}\me^{A_4(t-s)}A_3\widehat y(s)\d s
    \end{align*}
    and hence we obtain the same bound for
    $\Norm{\eta(t;\widehat t,\widehat \eta,\widehat y)}$ as for $\Norm{\eta(t;t^0,\eta^0,y)}$.
\end{example}


We are now in the position to summarize our results by showing initial and recursive feasibility of
the FMPC Algorithm~\ref{Algo:MPCFunnelCost} and proving Theorem~\ref{Th:FunnelMPCRelDeg1}.

\begin{proof}[Proof of Theorem~\ref{Th:FunnelMPCRelDeg1}]
    \emph{Step 1}:
    According to Theorem~\ref{Theorem:ExistanceBoundM}, there exists $M>0$
    satisfying~\eqref{eq:ExistanceBoundMInfinite} and~\eqref{eq:ExistanceBoundMRecursive}.
    Let $x^0\in B$ be an arbitrary initial value and $T\ge \delta$.
    Since $\Controls\neq \emptyset$ by~\eqref{eq:ExistanceBoundMInfinite},
 Theorem~\ref{Th:Funnel_cost_l2} yields the existence of some
 $u^{\star}\in \Controls$ such that~$J^{\phi}_T$ has a minimum, that is
    \[
        J^{\phi}_T(u^{\star};t^0,x^0)=\min_{
                \substack
                {
                    u\in L^{\infty}([t^0, t^0+T],\R^{m}),\\
                    \SNorm{u}  \leq M
                }
            }
             J^{\phi}_T(u;t^0,x^0),
    \]
    i.e., $u^{\star}$ is a solution of~\eqref{eq:FunnelMpcOCP} for $\widehat t = t^0$ and hence the FMPC
Algorithm~\ref{Algo:MPCFunnelCost} is initially  feasible. Furthermore, by $u^{\star}\in \Controls$
we have that the error satisfies $\Norm{e(t)} \le \phi(t)^{-1}$ for all $t\in [t^0,t^0+\delta)$.
\\

\noindent
    \emph{Step 2}:
Let $\widehat t\in t^0+\delta\N_0$ be such that the OCP~\eqref{eq:FunnelMpcOCP}
has a solution $u^{\star}\in \PhiTControls{\widehat t}{\widehat x}$ defined on $[\widehat t,\widehat
    t+ T]$ and let $x:[t^0,\widehat t+\delta)\to\R^n$ be the solution of~\eqref{eq:Sys}  under the FMPC feedback~\eqref{eq:FMPC-fb}.
    We now show that the OCP also has a solution at the next time step $\widehat t+\delta$. Since
    $u^{\star}$ is defined on $[\widehat t,\widehat t+ T]$, the solution $x$ has a continuous extension to $[\widehat t,\widehat
    t+ T]$ and, in particular, $\widehat x := x(\widehat t+\delta)$ is well defined. With $u_{\rm FMPC}(t) =
    \mu(t,x(\tilde t))$ for $t\in [\tilde t,\tilde t+\delta)$, $\tilde t\in t^0+\delta\N$, $\tilde
    t\le \widehat t$, the corresponding control input $u_{\rm FMPC}$ is well defined on $[t^0,\widehat
    t+\delta)$ and we have $x(t) = x(t;t^0,x^0,u_{\rm FMPC})$ for all $t\in[t^0,\widehat t+\delta)$.
    Then~\eqref{eq:ExistanceBoundMRecursive} gives that $\PhiTControls{\widehat t+\delta}{\widehat x}\neq
    \emptyset$ and by  Theorem~\ref{Th:Funnel_cost_l2} there exists $\tilde u\in \PhiTControls{\widehat
t+\delta}{\widehat x}$ such that~$J^{\phi}_T$ has a minimum, that is
    \[
        J^{\phi}_T(\tilde u;\widehat t+\delta,\widehat x)=\min_{
                \substack
                {
                    u\in L^{\infty}([\widehat t+\delta, \widehat t+\delta+T],\R^{m}),\\
                    \SNorm{u}  \leq M
                }
            }
             J^{\phi}_T(u;\widehat t+\delta,\widehat x),
    \]
    hence $\tilde u$ is a solution of~\eqref{eq:FunnelMpcOCP} on $[\widehat t+\delta, \widehat t+\delta+T]$.
    Under the feedback~\eqref{eq:FMPC-fb}, the solution~$x$ can thus be extended to $[t^0,\widehat
    t+2\delta)$ and, by definition of $\PhiTControls{\widehat t+\delta}{\widehat x}$, the corresponding tracking
    error $e$ satisfies $\Norm{e(t)} \le \phi(t)^{-1}$ for all $t\in [t^0,\widehat t+2\delta)$. This
    shows that the FMPC Algorithm~\ref{Algo:MPCFunnelCost} is recursively feasible.\\

\noindent
    \emph{Step 3}: By Step~2 we have shown that system~\eqref{eq:Sys} under the FMPC
    feedback~\eqref{eq:FMPC-fb} has a global solution $x:[t^0,\infty)\to\R^n$ and, since $u_{\rm
    FMPC}|_{[\widehat t,\widehat t+\delta]} \in \PhiControls{\delta}{\widehat t}{x(\widehat t)}$ for all $\widehat t\in
    t^0+\delta\N$, we have that~\ref{th:item:BoundedInput} and~\ref{th:item:ErrorInFunnel} hold.
\end{proof}





\section{Conclusion}\label{Sec:Conclusion}
In the present paper we have shown that  the FMPC scheme proposed in~\cite{berger2019learningbased},
which solves the problem of tracking a reference signal within a prescribed performance funnel, is
initially and recursively feasible for an arbitrary finite prediction horizon when applied to
nonlinear multi-input multi-output  systems with relative degree one and stable internal dynamics
(in the sense of a BIBS condition). By exploiting concepts from funnel control and using a new
``funnel-like'' stage cost function, feasibility is achieved without any need for additional
terminal or explicit output constraints while also being restricted to (a priori) bounded control
values. In particular, we have shown that the additional output constraints in the OCP of FMPC
considered in~\cite{berger2019learningbased} are not required to infer the feasibility results. We
have illustrated the  application of the FMPC scheme by a simulation not only of relative degree one
systems~-- for which feasibility is proved so far~-- but also of systems with higher relative
degree. The simulations  show promising preliminary results for  this  case, too. It  is a subject
of future research  to  show  that FMPC  is in fact applicable to a larger class of nonlinear
systems with  stable internal dynamics and higher relative degree.


\bibliographystyle{plain}
\footnotesize
\bibliography{\References}
\end{document}